\documentclass[11pt]{amsart}
\usepackage{amsmath}

\theoremstyle{plain}
\newtheorem{thrm}{Theorem}[section]
\newtheorem{lemma}[thrm]{Lemma}

\newtheorem{cor}[thrm]{Corollary}
\newtheorem{rmrk}[thrm]{Remark}

\numberwithin{equation}{section}

\begin{document}

\title[A Class of Mixed Type Equations]
{Smooth Solutions to a Class of\\
Mixed Type Monge-Amp\`{e}re Equations}
\author[Han]{Qing Han}
\address{Department of Mathematics\\
University of Notre Dame\\
Notre Dame, IN 46556} \email{qhan@nd.edu}

\address{Beijing International Center for Mathematical Research\\
Peking University\\
Beijing, 100871, China} \email{qhan@math.pku.edu.cn}

\author[Khuri]{Marcus Khuri}
\address{Department of Mathematics\\
Stony Brook University\\ Stony Brook, NY 11794}
\email{khuri@math.sunysb.edu}
\thanks{The first author acknowledges the support of NSF
Grant DMS-1105321. The second author acknowledges the support of
NSF Grant DMS-1007156 and a Sloan Research Fellowship.}
\begin{abstract}
We prove the existence of $C^{\infty}$ local solutions to a class of mixed type Monge-Amp\`{e}re equations
in the plane. More precisely, the equation changes type to finite order across two smooth curves intersecting transversely
at a point. Existence of $C^{\infty}$ global solutions to a corresponding class of linear mixed type equations
is also established. These results are motivated by, and may be applied to the problem of prescribed Gaussian curvature
for graphs, the isometric embedding problem for 2-dimensional Riemannian manifolds into Euclidean 3-space, and also
transonic fluid flow.
\end{abstract}

\maketitle

\section{Introduction}

In this paper, we will study a class of
Monge-Amp\`{e}re equations of mixed-type.
One source of interest in these equations arises from the equation of
prescribed Gaussian curvature.
Let $u$ be a
$C^2$ function defined in a domain $\Omega\subset \mathbb R^2$ and
suppose that the graph
of $u$ has Gaussian curvature $K(x)$ at the point $(x, u(x))$,
$x\in\Omega$. It follows
that $u$ satisfies the equation
\begin{equation*}\label{eq-GaussCur}
\det D^2u=K(x)(1+|Du|^2)^{2}.\end{equation*}
This equation is elliptic if $K$ is positive and hyperbolic if $K$ is negative,
and hence is of mixed type when $K$ changes sign.
Another source of Monge-Amp\`{e}re equations comes
from the isometric
embedding problem for 2-dimensional Riemannian manifolds into $\mathbb R^3$. See
chapter three in \cite{Han-Hong2006} for details. In \cite{Lin1986}, Lin
proved the existence of local isometric embeddings of surfaces into
$\mathbb R^3$ if the Gaussian curvature changes sign cleanly. In other
words, the Gaussian curvature changes sign to first order across a curve. In this
case the
Darboux equation, a basic equation associated with the isometric isometric embedding
problem is
elliptic on one side of the curve and hyperbolic on the other.
Such a result was generalized by the first named author in
\cite{Han2005} and \cite{Han200?}. Recently, we \cite{HanKhuri2011} discussed
a case in
which the Gaussian curvature changes sign in a more complicated way
and proved the existence of sufficiently smooth isometric embeddings. For further results
on this and related problems see [4]-[17], [21], and [22].

Mixed type equations also arise naturally in many other areas.
Recently, there have been several survey articles on this subject.
In \cite{Morawetz2004}, Morawetz gives a detailed account of the
historical background and known results on mixed type equations
and transonic flows. In \cite{Otway}, Otway presents a
detailed review on mixed type equations and Riemannian-Lorentzian metrics.
The most intensively studied equation of mixed type is the Tricomi
equation \cite{Tricomi} $$u_{yy}+yu_{xx}=f.$$  The plane is divided into two parts
by the $x$-axis. The Tricomi equation is elliptic in the upper
half plane and hyperbolic in the lower half plane. Many results have
been obtained in various settings for this equation. Nonetheless, beyond the equations
of the Tricomi family,
the theory of mixed type equations is fairly underdeveloped.
However this lack of development is not due to a lack of applications or
well-motivated problems. Mixed type equations which
change type in a way more complicated than that of the Tricomi case also
arise naturally in many circumstances. For instance, as far back as in 1929, Bateman
\cite{Bateman1929} presented several models for the 2-dimensional
motion of compressible fluids. One of these models is given by a
class of elliptic-hyperbolic equations in the unit
disk which change type in the following way. The unit disk is divided
into four regions by two straight lines through the origin.
These equations are elliptic in a pair of opposing regions and
hyperbolic in another pair of opposing regions. (See figure 1 on page 612 in \cite{Bateman1929}.)


In this paper, we study smooth solutions to a class of mixed type
Monge-Amp\`{e}re equations in the plane
which change type in a way similar to that in \cite{Bateman1929}.
The model equation has the following form
\begin{equation}\label{eq-MA}
u_{xx}u_{yy}-u_{xy}^2=(x^2-y^2)\psi(x, y, u, u_x, u_y), \end{equation}
where $\psi$ is a positive smooth function in
$B_1\times \mathbb R\times\mathbb R^2$. Here $B_1$ is the unit disk in
$\mathbb R^2$. We are interested in the question of whether or not (\ref{eq-MA}) admits a
\textit{smooth} solution $u$, defined in some neighborhood of the origin.
We note that (\ref{eq-MA}) is a Monge-Amp\`{e}re type equation of mixed type.
The unit ball $B_1\subset\mathbb R^2$ is
divided into four components by $\{|x|=|y|\}$. The equation (\ref{eq-MA})
is elliptic in $\{|x|>|y|\}$ and hyperbolic in $\{|x|<|y|\}$.

The following result is a special case of a more general result that we will prove in Section
\ref{Sec-Iterations}.

\begin{thrm}\label{Theorem-Nonlinear} Let $\psi$ be a positive smooth function in
$B_1\times \mathbb R\times\mathbb R^2$.  Then there exists a smooth solution
$u$ of (\ref{eq-MA}) in $B_r$ for some $r\in (0,1)$.
\end{thrm}

We should point out that $x^2-y^2$ can be replaced by any function with a similar
behavior, such as $y^2-x^2$. This is due to the invariance of the Monge-Amp\`{e}re
operator by orthogonal transformations.

In order to prove Theorem \ref{Theorem-Nonlinear}, it is essential to analyze
the corresponding linear equation. It turns out that it suffices to consider
\begin{equation}\label{0.1}u_{yy}+(x^2-y^2)u_{xx}=f.\end{equation} Again, the plane is
divided into four components by $\{|x|=|y|\}$. The equation (\ref{0.1})
is elliptic in $\{|x|>|y|\}$ and hyperbolic in $\{|x|<|y|\}$. The
lines of degeneracy $\{|x|=|y|\}$ are non-characteristic. Moreover,
the boundaries $\partial\{y>|x|\}$ and $\partial\{y<-|x|\}$ are
space-like for the corresponding hyperbolic regions $\{y>|x|\}$ and
$\{y<-|x|\}$, respectively. Hence equation (\ref{0.1}) is considerably more complicated than the
Tricomi equation, however we are still able to establish the following theorem, which is a special
case of a more general result proven in Section \ref{Section-Proof}.

\begin{thrm}\label{Theorem0.1} Let $f$ be a smooth function in
$\bar B_1\subset \mathbb R^2$. Then there exists a smooth solution
$u$ of (\ref{0.1}) in $B_1$.
 Moreover, for any positive
integer $s$, $u$ satisfies
\begin{equation}\label{0.2}\|u\|_{H^s(B_1)}\le
C_s\|f\|_{H^{s+5}(B_1)},\end{equation} where $C_s$ is a positive
constant depending only on $s$.\end{thrm}


We point out that (\ref{0.1}) is a small perturbation of the linearization for
(\ref{eq-MA}), at a suitably chosen approximate solution.
It should be emphasized that the form of the degenerate coefficient
$x^2-y^2$ plays an important role in the solvability of (\ref{0.1}).
If $x^2-y^2$ is replaced by other quadratic functions, then it may not be
possible to solve the new equation. For instance, the approach and methods used in this paper
do not yield solutions of
\begin{equation}\label{0.6}u_{yy}+(y^2-x^2)u_{xx}=f.\end{equation}
This equation is different from (\ref{0.1}), in that (\ref{0.6}) is elliptic in $\{|x|<|y|\}$
and hyperbolic in $\{|x|>|y|\}$. We note that
the $y$-direction, which may be considered as the time direction,
does not always point into
the hyperbolic regions. In this sense, the linear equation (\ref{0.1})
is more rigid than the nonlinear equation (\ref{eq-MA}).

The proof of Theorem \ref{Theorem0.1} consists of two steps. In the
first step, we construct a smooth solution in the elliptic regions
$\{|y|<x\}$ and $\{|y|<-x\}$. This is achieved by solving the
homogeneous Dirichlet problem. Such a solution then naturally yields
Cauchy data for the hyperbolic regions along the lines of degeneracy.
In the second step, we construct
a smooth solution in the hyperbolic regions $\{y>|x|\}$ and $\{y<-|x|\}$,
by solving the Cauchy problem. The solution constructed in Theorem
\ref{Theorem0.1} vanishes along the degenerate set $\{|x|=|y|\}\cap
B_1$. It is clear from the proof in this paper that one can prescribe
the solution arbitrarily (as a smooth function) on $\{|x|=|y|\}\cap B_1$. A similar
idea was used by Han \cite{Han2007} in the discussion of higher
dimensional Tricomi equations and related Monge-Amp\`{e}re
equations.

The difficulty in solving both the Dirichlet problem in the elliptic
regions and the Cauchy problem in the hyperbolic regions arises from two
distinct aspects of this problem. First, the equation is degenerate on the boundary. Second,
there is an angular point (i.e., the origin) on the boundary of each
domain.

Boundary value problems for (strictly) elliptic differential
equations in domains with angular points have been studied
extensively. The regularity results are in fact not encouraging. Well
known examples of harmonic functions in sector domains demonstrate that
these solutions are not necessarily smooth. Furthermore, in general, solutions of
degenerate elliptic differential equations exhibit worse
regularity than those of (strictly) elliptic differential equations.
Hence, it seems unrealistic to expect, at first glance, that solutions of
the degenerate elliptic equation studied here should
be smooth in domains with angular points. However, it is precisely due to the degeneracy at the
angular points that we are able to prove that the solutions have this high degree of regularity
up to the boundary. The {\it degeneracy} plays an important {\it
positive} role. In fact, we are not aware of any other cases where
degeneracy actually improves the regularity.

In contrast to the extensive studies of elliptic equations in nonsmooth
domains, little is known about the Cauchy problem for hyperbolic
equations when the initial curve is nonsmooth. Our first task here
is to prove that the Cauchy problem is well posed for (strictly)
hyperbolic equations in domains whose initial curves contain angular
points. Compatibility conditions are needed at the angular points in order
to ensure the regularity of solutions. (See Lemma
\ref{Lemma-extension0} for details.) As in the elliptic case, the
{\it degeneracy} along the initial curve surprisingly plays a {\it positive}
role in passing the existence and regularity result from strict
hyperbolicity to degenerate hyperbolicity. In fact, it demonstrates
that any such initial curve is space-like for the hyperbolic
regions. This plays an important role in the proof of the
well-posedness of degenerate hyperbolic equations in domains whose
initial curves have angular points.



This paper is organized as follows. In Section
\ref{Section-Elliptic}, we will construct smooth solutions for the
Dirichlet problem in the elliptic regions and derive necessary
estimates. Smooth solutions to the Cauchy problem for uniformly hyperbolic
equations in domains with angular points on the boundary will be established in
Section \ref{Section-HyperbolicExistence}. Estimates
independent of the hyperbolicity constant will then be derived in Section
\ref{Section-HyperbolicEstimates}. In Section
\ref{Section-Proof}, we will state and prove a general theorem of which
Theorem \ref{Theorem0.1} is a special case. Finally in Section
\ref{Sec-Iterations}, we will discuss a class of Monge-Amp\`{e}re type equations
and study the appropriate iterations to
prove a result which generalizes
Theorem \ref{Theorem-Nonlinear}.

\section{Elliptic Regions}\label{Section-Elliptic}

In this section, we will study a class of degenerate elliptic
differential equations in planar domains with angular
singularities. We will construct smooth solutions if the
degeneracy occurs at angular points.

For any $\kappa>0$, let $\mathcal C_\kappa$ be a cone in $\mathbb
R^2$ with vertex at the origin given by
$$\mathcal C_{\kappa}=\{(x,y); 0<|y|<\kappa x\}.$$
Let $\Omega_\kappa$ be a bounded domain in $\mathbb R^2$ such that
$$\Omega_\kappa\cap B_{1}=\mathcal C_{\kappa}\cap B_{1},$$
and
$$\partial\Omega_\kappa\setminus\{0\}\text{ is smooth.}$$ Consider the equation
\begin{equation}\label{4.1}
u_{yy}+Ku_{xx}+b_1u_x+b_2u_y+cu=f\quad\text{in
}\Omega_\kappa,\end{equation} where $K$, $b_i$ and $c$ are smooth
functions in $\bar\Omega_\kappa$. In the following, we assume
\begin{equation}\label{4.2}K> 0\ \text{in }\Omega_\kappa\quad\text{and}\quad K=0\ \text{on }
\partial\Omega_\kappa\cap B_1.\end{equation}
There are two major difficulties in studying (\ref{4.1}). First,
(\ref{4.1}) is degenerate on a portion of the boundary
$\partial\Omega_\kappa\cap B_1$. Second, there is an angular
singularity on the boundary. Usually, solutions of degenerate
elliptic differential equations exhibit a worse regularity than
those of (strictly) elliptic differential equations. It is well
known that solutions of (strictly) elliptic differential equations
in domains with angular singularities are in general not smooth. The
regularity depends on the angle in an essential way; the smaller the
angle, the better the regularity of solutions. However, it is
entirely different for equations which are degenerate at angular
points. In our case, we are able to construct smooth solutions of
(\ref{4.1}). Moreover, we can prove that any solutions of
(\ref{4.1}) are in fact smooth if its Dirichlet value on the
boundary is smooth and satisfies a compatibility condition up to
infinite order at the angular point. The {\it degeneracy} plays an
important {\it positive} role in the proof of the smoothness of
solutions at the angular point.

We will prove the following result.

\begin{thrm}\label{Theorem4.1} Let $K, b_i, c$ and $f$ be smooth
functions in $\bar{\Omega}_\kappa$ satisfying (\ref{4.2}), $c\le 0$
in $\Omega_\kappa$ and
\begin{equation}\label{4.3} |b_1|\le C_b(\sqrt{K}+|\partial_xK|)\quad\text{in
}\Omega_\kappa.\end{equation}  Then (\ref{4.1}) admits a smooth
solution in $\bar {\Omega}_\kappa$ with $u=0$ on
$\partial\Omega_\kappa$. Moreover, for any integer $m\ge 1$, $u$
satisfies
\begin{equation}\label{4.4}\|u\|_{H^m(\Omega_\kappa)}\le
C_m\|f\|_{H^{m+1}(\Omega_\kappa)},\end{equation} where $C_m$ is a
positive constant depending only on $C_b$ and the $C^{m}$-norms of
$K, b_i$ and $c$.\end{thrm}

To prove Theorem \ref{Theorem4.1}, we regularize (\ref{4.1}) by
replacing $K$ by $K+\delta$ for any $\delta>0$. Then the new
equation is uniformly elliptic and hence admits a unique solution
$u_\delta\in H^1_0(\Omega_\kappa)$. In order to pass limit as
$\delta\to0$, we need to derive estimates of $u_\delta$ independent
of $\delta$. The condition (\ref{4.3}) is introduced to overcome the
degeneracy of $K$ along $\partial\Omega_\kappa\cap B_1$.

In the following, we consider
\begin{equation}\label{2.1}\mathcal{L}u\equiv
u_{yy}+au_{xx}+b_1u_x+b_2u_y+cu=f\quad\text{in
}\Omega_\kappa,\end{equation} where $a$, $b_i$ and $c$ are smooth
functions in $\bar\Omega_\kappa$. We assume
\begin{equation}\label{3.1a}a_0\le a\le 1\quad\text{in }\Omega_\kappa,\end{equation}
for a positive
constant $a_0\in (0,1)$.

It is obvious that (\ref{2.1}) is uniformly elliptic. Hence
(\ref{2.1}) admits a solution $u\in H^1_0(\Omega_\kappa)$ and
classical results for uniformly elliptic differential equations on
smooth domains apply in any subdomains of $\bar{\Omega}_\kappa$
away from the origin. Specifically, for any $r\in(0,1)$ and any
$k\ge 2$, there holds
$$\|u\|_{H^k(\Omega_\kappa\setminus B_r)}\le
C_{k,r}\|f\|_{H^{k-2}(\Omega_\kappa)},$$ where $C_{k,r}$ is a
positive constant depending on $k$, $r$, $a_0$ and $C^{k-2}$-norms
of $a, b_i$ and $c$. In general, $C_{k,r}\to\infty$ as $r\to0$ or
$a_0\to0$. Therefore, we need to derive an estimate which is
independent of the lower bound of $a$. Moreover, the regularity of
$u$ close to the origin needs special attentions.

We first consider boundary points away from the origin.  We set for
any $\varepsilon>0$
\begin{equation}\label{3.0}
D_\varepsilon=\{(x,y); |x|< 1,\ 0<y<\varepsilon\}\subset \mathbb
R^2.\end{equation} We denote by $\partial^+_hD_\varepsilon$,
$\partial^-_hD_\varepsilon$ and $\partial_vD_\varepsilon$ the
horizontal top, horizontal bottom and vertical boundaries
respectively. By an appropriate transform, a neighborhood of any
point on $\partial\Omega_\kappa\setminus\{0\}$ is changed to
$D_\varepsilon$ for an $\varepsilon>0$. We consider (\ref{2.1}) in
$D_\varepsilon$ and assume \begin{equation}\label{3.3}|b_1|\le
C_b(\sqrt{a}+|\partial_xa|)\quad\text{in
}D_\varepsilon,\end{equation} for some positive constant $C_b$.

Lemma \ref{Thm3.1} and Corollary \ref{Cor3.1} below provide energy
estimates of solutions in narrow domains.

\begin{lemma}\label{Thm3.1}  Suppose $a,b_1,b_2$ and $c$ are smooth functions in
$D_\varepsilon$ satisfying (\ref{3.1a}) and (\ref{3.3}) and $u$ is a
smooth solution of (\ref{2.1}) with $u=0$ on
$\partial_h^-D_\varepsilon$. If
$$\varepsilon(
|c^+|_{L^\infty(D_\varepsilon)}+|a_{xx}|_{L^\infty(D_\varepsilon)}
+|b_{1,x}|_{L^\infty(D_\varepsilon)}
+|b_{2,y}|_{L^\infty(D_\varepsilon)}+1)^{\frac12}<1,$$
then for any cutoff function $\varphi=\varphi(x)$ on $(-1,1)$
\begin{align}\label{3.11}\begin{split}
&\|\varphi u\|_{L^2(D_\varepsilon)}+\|\varphi
u_y\|_{L^2(D_\varepsilon)}+\|\varphi \sqrt{a}
u_x\|_{L^2(D_\varepsilon)} \\
\leq &C_0 \big(
\|u\|_{L^2(\partial^+_hD_\varepsilon)}+\|u_y\|_{L^2(\partial^+_hD_\varepsilon)}
+\|\sqrt{\varphi a}u\|_{L^2(D_\varepsilon)}+\|\varphi f
\|_{L^2(D_\varepsilon)}\big), \end{split}\end{align} where $C_0$
is a positive constant depending only $\varphi$, $a, C_b$ and the
supnorm of $b_2$.
\end{lemma}

\begin{proof} For convenience, we set \begin{equation}\label{3.6}
M=|c^+|_{L^\infty(D_\varepsilon)}+|a_{xx}|_{L^\infty(D_\varepsilon)}
+|b_{1,x}|_{L^\infty(D_\varepsilon)}+|b_{2,y}|_{L^\infty(D_\varepsilon)}+1.
\end{equation}
Multiplying (\ref{2.1}) by $\varphi^2 u$ and integrating over
$D_\varepsilon$, we obtain
\begin{align}\label{3.12}\begin{split}\int_{D_\varepsilon}& (\varphi^2 u_y^2+\varphi^2
au_x^2)=\int_{D_\varepsilon}\big(\varphi^2
auu_x-\frac12(\varphi^2a)_xu^2+\frac12\varphi^2b_1u^2\big)_x\\
&+\int_{D_\varepsilon}\big(\varphi^2 uu_y+\frac12\varphi^2
b_2u^2\big)_y
+\int_{D_\varepsilon}\varphi^2\big(c+\frac12a_{xx}-\frac12b_{1,x}
-\frac12b_{2,y}\big)u^2\\
&+\int_{D_\varepsilon}\big((\varphi\varphi_{xx}+\varphi_x^2)a
+\varphi\varphi_x(2a_x-b_1)\big)u^2-\int_{D_\varepsilon} \varphi^2
uf.\end{split}\end{align} We first note that there is no boundary
integral over $\partial_v D_\varepsilon$ since $\varphi=0$ there and
there is no boundary integral on $\partial_h^-D_\varepsilon$ since
$u=0$ there. Next, we note that $\varphi_x^2\le C_\varphi\varphi$ on
$(-1,1)$ for some positive constant $C_\varphi$. Since $a\ge 0$ in
$D_1$, we also have
\begin{equation}\label{3.2}|\partial_xa|\le
C_a\sqrt{a}\quad\text{in supp}\varphi\times(-1,1),\end{equation} for
some positive constant $C_a$ depending only on supp$\varphi$ and the
$C^2$-norm of $a$. Then by (\ref{3.3}), (\ref{3.2}) and the Cauchy
inequality, we have
$$(\varphi\varphi_{xx}+\varphi_x^2)a
+\varphi\varphi_x(2a_x-b_1)\le \varphi^2+C_0\varphi a,$$ where $C_0$
is a positive constant depending only on $\varphi$, $C_a$ and $C_b$.
With (\ref{3.6}), we get
\begin{align*}\int_{D_\varepsilon} (\varphi^2 u_y^2+\varphi^2
au_x^2)\le&
C_0\int_{y=\varepsilon}\varphi^2(u^2+u_y^2)+C_0\int_{D_\varepsilon}\varphi au^2\\
&+ M\int_{D_\varepsilon}\varphi^2u^2 +\int_{D_\varepsilon} \varphi^2
f^2,\end{align*} where $C_0$ is a positive constant depending only
$\varphi$, $C_a, C_b$ and the supnorm of $b_2$. A simple integration
over $y$ yields
$$u^2(x,y)\le
y\int_{0}^\varepsilon u_y^2(x,t)dt,$$ and then
$$\int_{D_\varepsilon}\varphi^2u^2\le
\frac12\varepsilon^2\int_{D_\varepsilon}\varphi^2 u_y^2.$$ By a
simple substitution, we get
\begin{align*}\int_{D_\varepsilon} (\varphi^2 u_y^2+\varphi^2
au_x^2)\le&
C_0\int_{y=\varepsilon}\varphi^2(u^2+u_y^2)+C_0\int_{D_\varepsilon}\varphi au^2\\
&+ \frac12\varepsilon^2M\int_{D_\varepsilon}\varphi^2u_y^2
+\int_{D_\varepsilon} \varphi^2 f^2.\end{align*} With
$\varepsilon\sqrt{M}\le 1$, we then obtain
\begin{equation*}\int_{D_\varepsilon} (\varphi^2 u_y^2+\varphi^2
au_x^2)\le
C_0\big\{\int_{y=\varepsilon}\varphi^2(u^2+u_y^2)+\int_{D_\varepsilon}\varphi
au^2 + \int_{D_\varepsilon} \varphi^2 f^2\big\},\end{equation*} and
hence
\begin{equation*}\int_{D_\varepsilon} (\varphi^2u^2+\varphi^2 u_y^2+\varphi^2
au_x^2)\le
C_0\big\{\int_{y=\varepsilon}\varphi^2(u^2+u_y^2)+\int_{D_\varepsilon}\varphi
au^2 + \int_{D_\varepsilon} \varphi^2 f^2\big\}.\end{equation*} This
implies (\ref{3.11}) easily.
\end{proof}

\begin{cor}\label{Cor3.1}  Suppose $a,b_1,b_2$ and $c$ are smooth functions in
$D_\varepsilon$ satisfying (\ref{3.1a}) and (\ref{3.3}) and $u$ is a
smooth solution of (\ref{2.1}) with $u=0$ on
$\partial_h^-D_\varepsilon$. If for an integer $s\ge 1$,
$$s\varepsilon(
|c^+|_{L^\infty(D_\varepsilon)}+|a_{xx}|_{L^\infty(D_\varepsilon)}
+|b_{1,x}|_{L^\infty(D_\varepsilon)}
+|b_{2,y}|_{L^\infty(D_\varepsilon)}+1)^{\frac12}<1,$$
then for any cutoff function $\varphi=\varphi(x)$ on $(-1,1)$
\begin{equation}\label{3.16}\|\varphi u\|_{H^s(D_\varepsilon)}
\leq C_s \big(
\sum_{k=0}^{s+1}\|D^ku\|_{L^2(\partial_h^+D_\varepsilon)}
+\|u\|_{L^2(D_\varepsilon)}+\| f \|_{H^s(D_\varepsilon)}\big),
\end{equation} where $C_s$ is a positive constant depending
on $\varphi$, $C_b$ and the $C^s$-norms of $a, b_1, b_2$ and $c$.
\end{cor}

We emphasize that $C_s$ is independent of $\inf a$.

\begin{proof} We first claim for any integer $s\ge 0$
\begin{align}\label{3.17}\begin{split}
&\|\varphi \partial_x^su\|_{L^2(D_\varepsilon)}+\|\varphi
\partial_y\partial_x^su\|_{L^2(D_\varepsilon)}+\|\varphi \sqrt{a}
\partial_x^{s+1}u\|_{L^2(D_\varepsilon)} \\
\leq &C \big( \|\partial_x^su\|_{L^2(\partial_h^+D_\varepsilon)}
+\|\partial_y\partial_x^su\|_{L^2(\partial_h^+D_\varepsilon)}
+\|\sqrt{\varphi a}\partial_x^su\|_{L^2(D_\varepsilon)}\\
&+\sum_{k=0}^{s-1}\|\varphi
\partial_x^ku\|_{L^2(D_\varepsilon)}+\sum_{k=0}^{s-1}\|\varphi
\partial_y\partial_x^ku\|_{L^2(D_\varepsilon)}+\|\varphi\partial_x^s f
\|_{L^2(D_\varepsilon)}\big),
\end{split}\end{align}
where $C$ is a positive constant depending on $\varphi$, $C_b$ and
the $C^s$-norms of $a, b_1, b_2$ and $c$.

We first assume (\ref{3.17})  for any $s\ge 0$ and prove
(\ref{3.16}). By (\ref{3.17})$_{s-1}$ and (\ref{3.17})$_s$, with
different cutoff functions, we obtain
\begin{align*}
&\|\varphi
\partial_x^su\|_{L^2(D_\varepsilon)}+\|\varphi
\partial_y\partial_x^{s-1}u\|_{L^2(D_\varepsilon)}\\
\leq &C \big(
\sum_{k=s-1}^s\|\partial_x^ku\|_{L^2(\partial_h^+D_\varepsilon)}
+\sum_{k=s-1}^s\|\partial_y\partial_x^ku\|_{L^2(\partial_h^+D_\varepsilon)}
+\|\sqrt[4]{\varphi}\sqrt{a}\partial_x^{s-1}u\|_{L^2(D_\varepsilon)}\\
&+\sum_{k=0}^{s-1}\|\sqrt{\varphi}
\partial_x^ku\|_{L^2(D_\varepsilon)}+\sum_{k=0}^{s-2}\|\sqrt{\varphi}
\partial_y\partial_x^ku\|_{L^2(D_\varepsilon)}+\sum_{k=s-1}^s\|\sqrt{\varphi}\partial_x^k f
\|_{L^2(D_\varepsilon)}\big).
\end{align*}
Note by (\ref{2.1})
\begin{equation*}
\partial_{yy}u=-a\partial_{xx}u-b_1\partial_xu
-b_2\partial_yu-cu+f\quad\text{in }D_\varepsilon.\end{equation*} It
is obvious that derivatives of $u$ of order $s$ can be obtained
easily in terms of $\partial_x^su$ and lower order derivatives of
$u$. Hence we obtain
\begin{align*}
\sum_{i+j=s}&\|\varphi
\partial_x^i\partial_y^{j}u\|_{L^2(D_\varepsilon)}
\leq C \big(
\sum_{k=s-1}^s\|\partial_x^ku\|_{L^2(\partial_h^+D_\varepsilon)}
+\sum_{k=s-1}^s\|\partial_y\partial_x^ku\|_{L^2(\partial_h^+D_\varepsilon)}
\\
&+\sum_{i+j\le s-1}\|\sqrt{\varphi}
\partial_x^i\partial_y^ju\|_{L^2(D_\varepsilon)}
+\sum_{i+j\le s}\|\sqrt{\varphi}\partial_x^i\partial_y^j f
\|_{L^2(D_\varepsilon)}\big).
\end{align*}
This implies (\ref{3.16}) by a simple induction.

Next, we prove (\ref{3.17}). Applying $\partial_x^s$ to (\ref{2.1}),
we get
\begin{equation}\label{3.18}\partial_y^2\partial_x^su+a\partial_x^2\partial_x^su
+\tilde
b_1\partial_x\partial_x^su+b_2\partial_y\partial_x^su+\tilde
c\partial_x^su=f_s,\end{equation} where
\begin{align*}
\tilde b_1&=b_1+sa_x,\\
\tilde c&=c+s(b_1)_x+\frac12s(s-1)a_{xx},\end{align*} and
\begin{equation*}
f_s=\partial_x^sf+\sum_{i=0}^{s-1}\big(c_{s,i-2}\partial_x^{s-i+2}a
+c_{s,i-1}\partial_x^{s-i+1}b_1+c_{s,i}\partial_x^{s-i}c\big)\partial_x^iu
+\sum_{i=0}^{s-1}c_{s,i}\partial_x^{s-i}b_2\partial_y\partial_x^iu,\end{equation*}
where $c_{s,i}$ is a positive constant for $i=0, 1,\cdots, s-1$
with $c_{s,-2}=c_{s,-1}=0$.  Note that (\ref{3.18}) has the same
structure as (\ref{2.1}). So we can proceed as in the proof of
Lemma \ref{Thm3.1} to get an estimate of $\partial_x^su$. We only
need to note that the corresponding coefficient for
$\varphi^2(\partial_x^su)^2$, as compared with that for
$\varphi^2u^2$ in (\ref{3.12}), is given by
$$\tilde c+\frac12a_{xx}-\frac12\tilde
b_{1,x}-\frac12b_{2,y}=c+\frac12(s-1)^2a_{xx}+(s-\frac12)b_{1,x}-\frac12b_{2,y}.$$
By Lemma \ref{Thm3.1}, we have for $\varepsilon\le
(s\sqrt{M})^{-1}$
\begin{align*} &\|\varphi
\partial_x^su\|_{L^2(D_\varepsilon)}+\|\varphi
\partial_y\partial_x^su\|_{L^2(D_\varepsilon)}+\|\varphi \sqrt{a}
\partial_x^{s+1}u\|_{L^2(D_\varepsilon)} \\
\leq &C \big( \|\partial_x^su\|_{L^2(\partial_h^+D_\varepsilon)}
+\|\partial_y\partial_x^su\|_{L^2(\partial_h^+D_\varepsilon)}
+\|\sqrt{\varphi a}\partial_x^su\|_{L^2(D_\varepsilon)}\|+\|\varphi
f_s \|_{L^2(D_\varepsilon)}\big).\end{align*} With the explicit
expression of $f_s$, we get (\ref{3.17}) easily. \end{proof}

Next, we study solutions in a neighborhood of the origin. We first
recall some results for (strictly) elliptic differential equations
in domains with an angular singularity on the boundary. Main
references are \cite{Kondratev1967}, Chapter 4 and Chapter 5 in
\cite{Grisvard1985} or Chapter 6 in \cite{KozlovMazyaRoss1997}.

For any nonnegative integer $m$, define the space $V^m(\mathcal
C_\kappa)$ as the closure of $C_0^\infty(\bar{\mathcal
C}_\kappa\setminus \{0\})$ with respect to the norm
$$\|u\|_{V^m(\mathcal C_\kappa)}=\big(\sum_{|\alpha|\le
m}\int_{\mathcal C_\kappa}r^{2(|\alpha|-m)}|D^\alpha
u|^2\big)^{1/2}.$$

To illustrate how the regularity depends on the angle of the cone,
we consider
\begin{align}\label{1.1}\begin{split} \Delta u&=f\quad\text{in }\mathcal C_\kappa,\\
u&=0\quad\text{on }\partial\mathcal
C_\kappa.\end{split}\end{align} Let $\kappa=\tan(\alpha/2)$ for an
$\alpha\in (0,\pi)$. Obviously,
$u=\displaystyle{r^{\frac{\pi}{\alpha}}\cos({\pi\theta}/{\alpha})}$
is a solution of the homogeneous (\ref{1.1}). It is easy to check
that such a $u$ is in $V^m(\mathcal C_\kappa\cap B_1)$ provided
$${(m-1)\alpha}<{\pi}.$$
In general, the regularity $u\in V^m(\mathcal C_\kappa)$ cannot be
improved if $(m-1)\alpha/\pi$ is not an integer. Hence solutions
of (\ref{1.1}) exhibit a better regularity in smaller cones. This
turns out to be a general result.

We consider a slightly more general case. For a constant $a>0$, we
consider
\begin{align}\label{1.2}\begin{split} u_{yy}+au_{xx}&=f\quad\text{in }
\mathcal C_{\kappa},\\
u&=0\quad\text{on }\partial\mathcal
C_{\kappa}.\end{split}\end{align} By introducing $$x=\sqrt{a}
s,\quad y=t,$$ we have
\begin{align*} u_{tt}+u_{ss}&=f\quad\text{in }\mathcal C_{\sqrt{a}\kappa},\\
u&=0\quad\text{on }\partial\mathcal C_{\sqrt{a}\kappa}.\end{align*}

\begin{lemma}\label{Lemma1.1} Let $\kappa, a>0$ be constants
and $u\in H^1_0(\mathcal C_\kappa)$ be the unique solution of
(\ref{1.2}) for an $f\in L^2(\mathcal C_\kappa)$. Then for any
integer $m\ge 2$ satisfying
\begin{equation}\label{1.1a}{2(m-1)\arctan(\sqrt{a}\kappa)}<{\pi},\end{equation} if  $f\in
V^{m-2}(\mathcal C_\kappa)$, then $u$ is in $V^m(\mathcal
C_\kappa)$ and satisfies
$$\|u\|_{V^m(\mathcal C_\kappa)}\le C\|f\|_{V^{m-2}(\mathcal C_\kappa)},$$
where $C$ is a positive constant depending only on $m$, $a$ and
$\kappa$.\end{lemma}

Note that (\ref{1.1a}) always holds for $m=2$.

\begin{rmrk} If (\ref{1.1a}) is violated, then $u$ is not necessarily in
$V^m(\mathcal C_\kappa)$. To illustrate this, we consider
(\ref{1.1}), or (\ref{1.2}) with $a=1$. We write
$\kappa=\tan(\alpha/2)$ for an $\alpha\in (0,\pi)$ and let $m>2$
be an integer such that $(m-1)\alpha/\pi$ is not an integer. If
$f\in V^{m-2}(\mathcal C_\kappa)$, then any solution $u$ of
(\ref{1.1}) admits a decomposition
$$u=\sum_{j}c_jr^{\frac{j\pi}{\alpha}}\cos\frac{j\pi\theta}{\alpha}+w,$$
where $w\in V^m(\mathcal C_\kappa)$ and the summation is extended
over all integer $j$ in the interval $(\alpha/\pi,
(m-1)\alpha/\pi)$.
\end{rmrk}

For solutions of (\ref{2.1}), the regularity is governed by the
corresponding result for the constant coefficient operator
$\partial_{yy}+a(0)\partial_{xx}$.

\begin{lemma}\label{Lemma1.2} Let $\kappa$ be a constant,
$a,b_1,b_2$ and $c$ be smooth functions in $\Omega_\kappa$
satisfying (\ref{3.1a}) and $u\in H^1_0(\Omega_\kappa)$ be a
solution of (\ref{2.1}) for an $f\in L^2(\Omega_\kappa)$. Then for
any integer $m\ge 2$ satisfying
\begin{equation}\label{1.1b}{2(m-1)\arctan(\sqrt{a(0)}\kappa)}<{\pi},\end{equation}
if  $\xi f\in V^{m-2}(\mathcal C_\kappa)$, then $\eta u$ is in
$V^m(\mathcal C_\kappa)$ and satisfies
$$\|\eta u\|_{V^m(\mathcal C_\kappa)}\le C(\|\xi u\|_{L^2(\Omega_\kappa)}+
\|\xi f\|_{V^{m-2}(\mathcal C_\kappa)}),$$ where $\xi$ and $\eta$
are two arbitrary cutoff functions in $B_1$ with $\xi=1$ on the
support of $\eta$ and $C$ is a positive constant depending only on
$m$, $a(0)$, $\kappa$, $\xi$, $\eta$ and $C^{m-2}$-norms of $a,
b_i$ and $c$.\end{lemma}

Later on, we will only use the regularity assertion, instead of
estimates, in Lemma \ref{Lemma1.2}. Now we begin to derive
estimates of $u$ close to the origin independent of $\inf a$. The
main result for this part is the following lemma.

\begin{lemma}\label{Lemma2.41} Let $m$ be an integer, $f\in H^m(\mathcal C_\kappa\cap B_1)$
and $u$ be an $H^1$-solution of (\ref{2.11}) in $\mathcal C_1\cap
B_1$ satisfying $u=0$ on $\theta=\pm\alpha$. Then there exist
constants $\delta_m$ and $\kappa_m$ such that, if
$\kappa\le\kappa_m$ and $\sqrt{a(0)}\kappa\le\delta_m$, then $u\in
H^{m}(\mathcal C_1\cap B_1)$ and
\begin{equation}\label{2.40}
\|u\|_{H^m(\mathcal C_\kappa\cap \{x<\frac12\})}\le
C_m\big(\sum_{i=0}^{m+1}\|D^iu\|_{L^2(\mathcal C_\kappa\cap
\{x=\frac12\})}+\|f\|_{H^m(\mathcal C_\kappa\cap B_1)}\big),
\end{equation} where $\delta_m$ is a positive constant depending only on $m$,
$\kappa_m$ is a positive constant depending only on the $C^2$-norms
of $a, b_i, c$ and $C_m$ is a positive constant depending only on
the $C^{m}$-norms of $a$, $b_i$ and $c$.
\end{lemma}

We emphasize  that $C_m$ is independent of $\inf a$. The proof of
Lemma \ref{Lemma2.41} is complicated. We first establish some
lemmas.

\begin{lemma}\label{Lemma2.40} Let $m$ be a nonnegative integer,
$a,b_1,b_2$ and $c$ be smooth functions in $\Omega_\kappa$
satisfying (\ref{3.1a}) and $u$ be an $H^2$-solution of (\ref{2.1})
with $u=0$ on $\partial\Omega_\kappa\cap B_1$. Then there exists a
positive constant $\eta_m$ depending only on $m$ such that, if
\begin{equation}\label{2.2a}\kappa^2a(0)\le \eta_m,\end{equation} then there exists
a polynomial $\mathcal P_m(u)$ of degree $m$ such that $\mathcal
P_m(u)=0$ on $\partial\Omega_\kappa\cap B_1$, any coefficient $c_k$
in the homogeneous part of degree $k$, for any $k\le m$, in
$\mathcal P_m(u)$ satisfies
$$|c_k|\le C_k\sum_{|\alpha|\le k-2}|D^\alpha f(0)|,$$ and
\begin{equation*}\mathcal L\big(u-\mathcal{P}_m(u)\big)=\text{ $(m-2)$-th remainder
of }f- \tilde{\mathcal L}(\mathcal P_{m}(u)),\end{equation*}  where
$C_k$ is a positive constant depending only on $\delta_m$, $\kappa$,
and the $C^{k-1}$-norms of $a, b_i$ and $c$, and $\tilde{\mathcal
L}=(a-a(0))\partial_{xx}+b_1\partial_x+b_2\partial_y+c$.
\end{lemma}

It is easy to see from the proof below that $\mathcal{P}_m(u)$ is
the $m$-th Taylor polynomial of $u$ at 0 if $u$ is $C^m$ in a
neighborhood of the origin.

\begin{proof} We first consider the transform
$(x,y)\mapsto (x/\kappa,y)$. Then (\ref{2.1}) has the form
\begin{equation*}
u_{yy}+\kappa^2au_{xx}+\kappa b_1u_x+b_2u_y+cu=f\quad\text{in
}\tilde\Omega_1,\end{equation*} for a domain
$\tilde\Omega_1\subset\mathbb R^2$ with $\tilde\Omega_1\cap
B_1=\mathcal C_1\cap B_1$. In the following, we simply assume
$\kappa=1$.

Let $\mathcal L_0=\partial_{yy}+a(0)\partial _{xx}$. We first note
$\partial_xu(0)=\partial_yu(0)=0$ since $u=0$ on $y=\pm\kappa x$.
Set
$$u=\mathcal{P}_m(u)+\mathcal{R}_m(u)=Q_2+\cdots+Q_{m}+\mathcal{R}_m(u),$$
where $Q_{k}$ is a homogeneous polynomial of degree $k$ for
$k=2,\cdots, m$, with $Q_2=\cdots=Q_{m}=0$ on
$\partial\Omega_1\cap B_1$. Hence for $k=2, \cdots, m$, $Q_k$ has
the form
$$Q_{k}=(x^2-y^2)\sum_{i=0}^{k-2}c_{k-2,i}x^{k-2-i}y^i.$$
Then
$$\mathcal L\big(\mathcal{R}_m(u)\big)+\mathcal L_0Q_2+\cdots+\mathcal L_0Q_{m}=\tilde
f,$$ where $$\tilde f=f-(\mathcal L-\mathcal
L_0)\big(\mathcal{P}_m(u)\big)=f-\sum_{k=2}^{m}(\mathcal
L-\mathcal L_0)Q_{k}.$$ Note that $\mathcal L_0Q_k$ is a
homogeneous polynomial of degree $k-2$. We set
\begin{equation}\label{2.2}\mathcal L_0Q_{k}=\text{the $(k-2)$-th
homogeneous part of }\tilde f,\quad\text{for each }k=2,\cdots,
m.\end{equation} Then
\begin{equation*}\mathcal L\big(\mathcal{R}_m(u)\big)=\text{ $(m-2)$-th remainder
of }f- (\mathcal L-\mathcal L_0)(\mathcal
P_{m}(u)).\end{equation*} We claim that we can solve successively
$Q_2, Q_3, \cdots, Q_{m}$. In fact, a simple calculation shows
\begin{align*} \mathcal
L_0Q_{k}=&\sum_{i=0}^{k-2}\big((k-i)(k-i-1)a(0)-(i+2)(i+1)\big)c_{k-2,i}\\
&-\sum_{i=2}^{k-2}(k-i)(k-i-1)a(0)c_{k-2,i-2}+\sum_{i=0}^{k-4}(i+2)(i+1)c_{k-2,i+2}.
\end{align*}
If we write (\ref{2.2}) as a linear system for $c_{k-2,0},
c_{k-2,1},\cdots, c_{k-2,k-2}$, the $(k-1)\times(k-1)$ coefficient
matrix  is obviously invertible if $a(0)=0$ and hence invertible if
$a(0)$ is small. It is easy to see that $c_{k-2,0},
c_{k-2,1},\cdots, c_{k-2,k-2}$ solving (\ref{2.2}) is a linear
combination of $D^\alpha f(0)$, $|\alpha|\le k$.
\end{proof}

To discuss the regularity of solutions close to the origin, we need
to consider (\ref{2.1}) in polar coordinates. We note
$$x=r\cos\theta, \quad y=r\sin\theta.$$
It is easy to see that any $D^\alpha u$, for some $|\alpha|=m$, is a
linear combination of
$$\frac{1}{r^{m-i}}\partial_r^i\partial_\theta^j u,\quad 1\le
i+j\le m,$$ with coefficients given by smooth functions of $\theta$.

In polar coordinates, (\ref{2.1}) has the form
\begin{equation}\label{2.11}\tilde{a}_{11}r^2u_{rr}+2\tilde{a}_{12}ru_{r\theta}+
u_{\theta\theta}+\tilde{b}_1ru_r+\tilde{b}_2u_\theta+\tilde
cu=\tilde f,\end{equation} where
\begin{align*}
\tilde{a}_{11}=&\frac{\sin^2\theta+a\cos^2\theta}{\cos^2\theta+a\sin^2\theta},\\
\tilde{a}_{12}=&\frac{(1-a)\sin\theta\cos\theta}{\cos^2\theta+a\sin^2\theta},\\
\tilde{b}_1=&\frac{\cos^2\theta+a\sin^2\theta
+b_1r\cos\theta+b_2r\sin\theta}{\cos^2\theta+a\sin^2\theta},\\
\tilde{b}_2=&\frac{-2(1-a)\cos\theta\sin\theta
-b_1r\sin\theta+b_2r\cos\theta}{\cos^2\theta+a\sin^2\theta},\\
\tilde c=&\frac{r^2c}{\cos^2\theta+a\sin^2\theta},\end{align*} and
$$\tilde f=\frac{r^2f}{\cos^2\theta+a\sin^2\theta}.$$

\begin{lemma}\label{Lemma2.11} Let $u$ be a $C^m$-solution of
(\ref{2.11}). Then for any integer $k,l$ with $1\le k+l\le m$
$$\frac{1}{r^{m-k}}\partial_r^k\partial_\theta^l
u=\sum_{i=0}^{m}\frac{1}{r^{m-i}}\big(c_{i}\partial_r^iu+d_{i}
\partial_r^i\partial_\theta
u\big)
+\sum_{i=2}^m\sum_{j=0}^{m-i}\frac{1}{r^{m-i}}e_{ij}\partial_r^{i-2}
\partial_\theta^j\tilde
f,$$ where $d_{m}=0$ and all coefficients $c_{i}, d_{i}$ and
$e_{ij}$ are functions depending on derivatives of $\tilde{a}_{11},
\tilde{a}_{12}, \tilde{b}_1, \tilde{b}_2$ and $\tilde{c}$ (with
respect to $r$ and $\theta$) up to the order $m-2$.
\end{lemma}

The proof is by a simple induction based on (\ref{2.11}) and hence
omitted. Therefore, in order to estimate $D^mu$, we only need to
estimate
$$\frac{1}{r^{m-i}}\partial_r^i\partial_\theta^ju,\quad 0\le j\le 1,\ 0\le i+j\le m.$$

In the following, we assume $\kappa\le 1$ and consider (\ref{2.1})
in
$$R=\{(x,y)\in\Omega_\kappa; x<1/2\},$$
or the equivalent (\ref{2.11}) in
$$R=\{(r,\theta); 0<r<r(\theta), -\alpha<\theta<\alpha\},$$
where $\alpha\in(0,\pi/2)$ with $\tan\alpha=\kappa$ and
$r=r(\theta)$ corresponds to $x=1/2$, hence
$r(\theta)=1/(2\cos\theta)$.

\begin{lemma}\label{Lemma2.16} Let $\mu$ be a positive constant and $u$ be a $C^2$-solution
of (\ref{2.1}) in $R$ satisfying $u=0$ on $\theta=\pm\alpha$ and
$$\int_R\big(\frac{u^2}{r^\mu}+\frac{u_r^2}{r^{\mu-2}}\big)<\infty.$$
Then there exists a sufficiently small $\kappa_0$ such that, if
$$\kappa(|a|_{C^2}+|b_i|_{C^1}+|c|+1)^{\frac12}<\kappa_0,$$ then
\begin{equation}\label{2.15}
\int_R\big(\frac{u^2}{r^\mu}
+\frac{1}{r^{\mu-2}}(u_r\sin\theta+\frac1ru_\theta\cos\theta)^2\big)
\le C_0\int_{\partial_v^+R}(u^2+u_r^2)+\int_R\frac{f^2}{r^{\mu-4}},
\end{equation}
where $C_0$ is a positive constant depending only on the $C^1$-norm
of $a$ and the $L^\infty$-norm of $b_i$.
\end{lemma}

The proof is similar to that of Lemma \ref{Thm3.1}.

\begin{proof}
We multiply (\ref{2.11}) by $-u/r^\mu$ and get by a straightforward
calculation
\begin{align}\label{2.16}\begin{split}
&-\frac1{2r}\big(\frac{2\tilde{a}_{11}uu_r}{r^{\mu-3}}
-\big(2\tilde{a}_{12,\theta}+r\tilde{a}_{11,r}
-(\mu-3)\tilde{a}_{11}-\tilde{b}_1\big)\frac{u^2}{r^{\mu-2}}\big)_r\\
&\quad-\big(\frac{uu_\theta}{r^\mu}+\frac{2\tilde{a}_{12}uu_r}{r^{\mu-1}}
+\frac{\tilde{b}_2u^2}{2r^\mu}\big)_\theta+ \frac{1}{r^{\mu-2}}\big(
\frac{u_\theta^2}{r^2}+\frac{2\tilde{a}_{12}u_\theta
u_r}{r}+\tilde{a}_{11}u_r^2\big)=\Lambda
\frac{u^2}{r^\mu}-\frac{u\tilde f}{r^{\mu}},
\end{split}\end{align} where
\begin{align}\label{2.16a}\begin{split}\Lambda=&\frac12
\big(2r\tilde{a}_{12,\theta
r}-2(\mu-2)\tilde{a}_{12,\theta}+r^2\tilde{a}_{11,rr}
-2(\mu-3)r\tilde{a}_{11,r}\\
&+(\mu-3)(\mu-2)\tilde{a}_{11}
-r\tilde{b}_{1,r}+(\mu-2)\tilde{b}_1-\tilde{b}_{2,\theta}+2c\big).
\end{split}\end{align} A simple calculation shows
\begin{align*}
&\frac{u_\theta^2}{r^2}+\frac{2\tilde{a}_{12}u_\theta
u_r}{r}+\tilde{a}_{11}u_r^2\\
=&\frac{1}{\cos^2\theta+a\sin^2\theta}\big((u_r\sin\theta+
\frac1ru_\theta\cos\theta)^2
+a(u_r\cos\theta-\frac1ru_\theta\sin\theta)^2\big)\\
\ge& (u_r\sin\theta+\frac1ru_\theta\cos\theta)^2
+a(u_r\cos\theta-\frac1ru_\theta\sin\theta)^2=u_y^2+au_x^2.\end{align*}
Now we integrate (\ref{2.16}) with respect to $rdrd\theta$ in
$$R_{\bar r}=\{(r,\theta); \bar r<r<r(\theta), -\alpha<\theta<\alpha\},$$
for any $\bar r<1/2$. Since $u=0$ on $\theta=\pm\alpha$, there is no
boundary integral on $\theta=\pm\alpha$. By the Cauchy inequality,
we have
\begin{align*}
\int_{R_{\bar r}}\frac{1}{r^{\mu-2}}(u_y^2+au_x^2)\le&
C_0\int_{-\alpha}^\alpha\big(\frac{u^2}{r^\mu}+\frac{u_r^2}{r^{\mu-2}}\big)
\big|_{r=\bar r}d\theta+C_0\int_{r=r(\theta)}(u^2+u_r^2)\\
&+\int_{R_{\bar r}}(\Lambda+1) \frac{u^2}{r^\mu}+\int_{R_{\bar
r}}\frac{\tilde f^2}{r^{\mu}},
\end{align*}
where $C_0$ depends on the $L^\infty$-norms of
$\tilde{a}_{12,\theta}$, $r\tilde{a}_{11,r}$, $\tilde{a}_{11}$ and
$\tilde{b}_1$. Next, we write $$\int_{R\cap
\{r<\frac12\}}\big(\frac{u^2}{r^{\mu}}+\frac{u_r^2}{r^{\mu-2}}\big)=\int_0^{\frac12}
\frac1r\int_{-\alpha}^\alpha
\big(\frac{u^2}{r^{\mu-2}}+\frac{u_r^2}{r^{\mu-4}}\big)d\theta dr.$$
Then there exists a sequence $r_i\to 0$ such that
$$\int_{-\alpha}^\alpha
\big(\frac{u^2}{r^{\mu-2}}+\frac{u_r^2}{r^{\mu-4}}\big)
\big|_{r=r_i}d\theta\to0\quad\text{as
}r_{i}\to0. $$ By taking $\bar r=r_i\to0$, we have
\begin{equation}\label{2.17}
\int_R\frac{1}{r^{\mu-2}}(u_y^2+au_x^2) \le
C_0\int_{r=r(\theta)}(u^2+u_r^2)+\int_R(\Lambda+1)
\frac{u^2}{r^\mu}+\int_R\frac{\tilde f^2}{r^{\mu}}.
\end{equation}

For any $(r, \theta)\in R$, the corresponding $(x,y)$ satisfies
$|y|<\kappa x$ and $x<1/2$. Since $u(x, -\kappa x)=0$, we have
$$u(x,y)=\int_{-\kappa x}^y u_y(x,t)dt,$$
and $$u^2(x,y)\le 2\kappa x\int_{-\kappa x}^{\kappa
x}u_y^2(x,t)dt.$$ Note that $x^2\le r^2=x^2+y^2\le (\kappa^2+1)x^2.$
This implies
$$\frac{u^2(x,y)}{(x^2+y^2)^{\frac\mu2}}\le 2(\kappa^2+1)^{\frac{\mu-2}2}
\frac{\kappa x}{x^2+y^2}\int_{-\kappa x}^{\kappa
x}\frac{u_y^2(x,t)}{(x^2+t^2)^{\frac{\mu-2}{2}}}dt.$$ An integration
in $R=\{(x,y)\in\Omega_\kappa; x<1/2\}$ yields
$$\int_R\frac{u^2}{r^\mu}\le 4\kappa^2
(\kappa^2+1)^{\frac{\mu-2}2}\int_R\frac{u_y^2}{r^{\mu-2}}.$$ If
$\kappa$ is small so that
$$4\kappa^2
(\kappa^2+1)^{\frac{\mu-2}2}(|\Lambda|_{L^\infty}+1)\le \frac12,$$
we then have by (\ref{2.17})
\begin{equation*}
\int_R\big(\frac{u^2}{r^\mu}+\frac{u_y^2}{r^{\mu-2}}+\frac{au_x^2}{r^{\mu-2}}\big)
\le C_0\int_{r=r(\theta)}(u^2+u_r^2)+\int_R\frac{\tilde
f^2}{r^{\mu}}.
\end{equation*}
This implies (\ref{2.15}). \end{proof}

\begin{lemma}\label{Lemma2.21} Let $m$ be an integer and $u$ be a $C^{m+2}$-solution
of (\ref{2.11}) in $R$ satisfying $u=0$ on $\theta=\pm\alpha$ and
$$\int_R\sum_{i=0}^{m+1}\frac{(\partial_r^iu)^2}{r^{2(m-i)}}
<\infty.$$ If
\begin{equation}\label{2.19}m\kappa(|a|_{C^{2}}
+|b_i|_{C^{1}}+|c|_{L^\infty}+1)^{\frac12}<\kappa_0,\end{equation}
then
\begin{equation}\label{2.20}
\int_R\big(\sum_{i=0}^m\frac{(\partial_r^iu)^2}{r^{2(m-i)}}
+\sum_{i=0}^{m-1}\frac{(\partial_\theta\partial_r^iu)^2}{r^{2(m-i)}}\big)
\le C_m\int_{\partial_v^+R}\sum_{i=0}^{m+1}(\partial_r^iu)^2
+\int_R\sum_{i=0}^{m}\frac{(\partial_r^if)^2}{r^{2(m-i-2)^+}},
\end{equation} where $\kappa_0$ is as in Lemma \ref{Lemma2.16}
and $C_m$ is a positive constant depending only on the
$C^{m}$-norms of $a$, $b_i$ and $c$.
\end{lemma}

\begin{proof} For any $s=0,1\cdots, m$, we apply $\partial_r^s$ to
(\ref{2.11}) to get
\begin{align}\label{2.21}\begin{split}\tilde{a}_{11}r^2(\partial_r^su)_{rr}
+&2\tilde{a}_{12}r(\partial_r^su)_{r\theta}+
(\partial_r^su)_{\theta\theta}+\tilde{b}^{(s)}_1r(\partial_r^su)_r
+\tilde{b}^{(s)}_2(\partial_r^su)_\theta\\
&+\tilde c^{(s)}\partial_r^su=\tilde f^{(s)}-\tilde
d^{(s)}\partial_\theta\partial_r^{s-1}u,\end{split}\end{align} where
\begin{align*}
\tilde{b}_1^{(s)}=&\tilde{b}_1+s(r\tilde{a}_{11,r}+2\tilde{a}_{11}),\\
\tilde{b}_2^{(s)}=&\tilde{b}_{2}+2s(r\tilde{a}_{12})_r,\\
\tilde{c}^{(s)}=&c+s\partial_r(r\tilde{b}_1)+\frac12s(s-1)(r^2\tilde{a}_{11})_{rr},\\
\tilde{d}^{(s)}=&s\tilde{b}_{2,r}+s(s-1)(r\tilde{a}_{12})_{rr},\end{align*}
and \begin{align*}\tilde
f^{(s)}=&\partial_r^s(r^2f)-\sum_{i=0}^{s-2}\big(2c_{s,i-1}
\partial_r^{s-i+1}(r\tilde{a}_{12})
+c_{s,i}\partial_r^{s-i}\tilde{b}_2\big)\partial_\theta\partial_r^iu\\
&-\sum_{i=0}^{s-1}\big(c_{s,i-2}\partial_r^{s-i+2}(r^2\tilde{a}_{11})
+c_{s,i-1}\partial_r^{s-i+1}(r\tilde{b}_1)
+c_{s,i}\partial_r^{s-i}c\big)\partial_r^iu,\end{align*} where
$c_{s,i}$ is a constant depending only on $s$ and $i$ with
$c_{s,-2}=c_{s,-1}=0$. Since (\ref{2.21}) has a similar structure as
(\ref{2.11}), we may apply Lemma \ref{Lemma2.16} to (\ref{2.21}). If
$$\int_R(\frac{(\partial_r^su)^2}{r^{\mu}}+
\frac{(\partial_r^{s+1}u)^2}{r^{\mu-2}})<\infty,$$
and
$$4\kappa^2
(\kappa^2+1)^{\frac{\mu-2}2}(|\Lambda_s|_{L^\infty}
+|\tilde{d}^{(s)}|_{L^\infty}+1)\le
\kappa_0,$$ we obtain
\begin{align*}
&\int_R\big(\frac{(\partial_r^su)^2}{r^\mu}+\frac{1}{r^{\mu-2}}
(\partial_r^{s+1}u\sin\theta+\frac1r\partial_\theta\partial_r^su\cos\theta)^2\big)\\
\le&
C_0\int_{\partial_v^+R}((\partial_r^su)^2+(\partial_r^{s+1}u)^2)
+s^2\int_R\frac{(\partial_\theta\partial_r^{s-1}u)^2}{r^{\mu}}
+\int_R\frac{(\tilde{f}^{(s)})^2}{r^{\mu}},
\end{align*}
where $\Lambda_s$ is as $\Lambda$ in (\ref{2.16a}) with
$\tilde{a}_{ij}, \tilde{b}_i, \tilde{c}$ replaced by
$\tilde{a}_{ij}^{(s)}, \tilde{b}_i^{(s)}, \tilde{c}^{(s)}$. For each
$s=0,1,\cdots, m$, we take $\mu=2(m-s)$ and then obtain
\begin{align*}
&\int_R\big(\frac{(\partial_r^su)^2}{r^{2(m-s)}}+\frac{1}{r^{2(m-s-1)}}
(\partial_r^{s+1}u\sin\theta+\frac1r\partial_\theta\partial_r^su\cos\theta)^2\big)\\
\le&
C_0\int_{\partial_v^+R}((\partial_r^su)^2+(\partial_r^{s+1}u)^2)
+s^2\int_R\frac{(\partial_\theta\partial_r^{s-1}u)^2}{r^{2(m-s)}}
+\int_R\frac{(\tilde{f}^{(s)})^2}{r^{2(m-s)}}.
\end{align*}
Note that
$$(\tilde{f}^{(s)})^2\le
C_s\big(\sum_{i=s-2}^sr^{2(i-s+2)}(\partial_r^if)^2
+\sum_{i=0}^{s-1}(\partial_r^iu)^2
+\sum_{i=0}^{s-2}(\partial_\theta\partial_r^iu)^2\big),$$
where $C_s$ depends on the $C^s$-norms of $\tilde{a}_{ij},
\tilde{b}_i$ and $\tilde{c}$. Hence we have
\begin{equation}\label{2.24}\begin{split}
&\int_R\big(\frac{(\partial_r^su)^2}{r^{2(m-s)}}+\frac{1}{r^{2(m-s-1)}}
(\partial_r^{s+1}u\sin\theta+\frac1r\partial_\theta\partial_r^su\cos\theta)^2\big)\\
\le&
C_0\int_{\partial_v^+R}((\partial_r^su)^2+(\partial_r^{s+1}u)^2)
+s^2\int_R\frac{(\partial_\theta\partial_r^{s-1}u)^2}{r^{2(m-s+1)}}\\
&+\int_R\big(\sum_{i=s-2}^s\frac{(\partial_r^if)^2}{r^{2(m-i-2)}}
+\sum_{i=0}^{s-1}\frac{(\partial_r^iu)^2}{r^{2(m-i)}}
+\sum_{i=0}^{s-2}\frac{(\partial_\theta\partial_r^iu)^2}{r^{2(m-i)}}\big).
\end{split}\end{equation}

Now we claim for any $k=0,1\cdots,m$
\begin{equation}\label{2.25}\begin{split}
&\int_R\big(\frac{(\partial_r^ku)^2}{r^{2(m-k)}}
+\frac{(\partial_\theta\partial_r^{k-1}u)^2}{r^{2(m-k+1)}}\big)\\
\le& C_s\big(\int_{\partial_v^+R}\sum_{i=0}^{k+1}(\partial_r^iu)^2
+\int_R\sum_{i=0}^k\frac{(\partial_r^if)^2}{r^{2(m-i-2)^+}}\big).
\end{split}\end{equation}
Note that (\ref{2.25})$_0$ is simply a part of (\ref{2.24})$_0$. Now
we assume that (\ref{2.25}) holds for $k=0, 1, \cdots, s\le m-1$ and
prove (\ref{2.25}) for $k=s+1$. By (\ref{2.24})$_s$ and
(\ref{2.25})$_0$, $\cdots$, (\ref{2.25})$_{s}$, we have
\begin{align*}
&\int_R\frac{1}{r^{2(m-s-1)}}
(\partial_r^{s+1}u\sin\theta+\frac1r\partial_\theta\partial_r^su\cos\theta)^2\\
\le& C_s\big(\int_{\partial_v^+R}\sum_{i=0}^{k+1}(\partial_r^iu)^2
+\int_R\sum_{i=0}^k\frac{(\partial_r^if)^2}{r^{2(m-i-2)^+}}\big),
\end{align*}
or
\begin{equation}\label{2.26}\begin{split}
\int_R\frac{1}{r^{2(m-s)}} &(\partial_\theta\partial_r^su)^2
\le\kappa^2\int_R\frac{1}{r^{2(m-s-1)}} (\partial_r^{s+1}u)^2\\&+
C_s\big(\int_{\partial_v^+R}\sum_{i=0}^{k+1}(\partial_r^iu)^2
+\int_R\sum_{i=0}^k\frac{(\partial_r^if)^2}{r^{2(m-i-2)^+}}\big),
\end{split}\end{equation}
where we used $|\tan\theta|\le \tan\alpha=\kappa$ for any
$|\theta|<\alpha$. Next, by (\ref{2.24})$_{s+1}$ and
(\ref{2.25})$_0$, $\cdots$, (\ref{2.25})$_{s}$, we have
\begin{equation}\label{2.27}\begin{split}
\int_R\frac{1}{r^{2(m-s-1)}} &(\partial_r^{s+1}u)^2
\le (s+1)^2\int_R\frac{1}{r^{2(m-s)}} (\partial_\theta\partial_r^su)^2\\
&+C_{s+1}\big(\int_{\partial_v^+R}\sum_{i=0}^{s+2}(\partial_r^iu)^2
+\int_R\sum_{i=0}^{s+1}\frac{(\partial_r^if)^2}{r^{2(m-i-2)^+}}\big).
\end{split}\end{equation}
If $\kappa(s+1)<1/2$, (\ref{2.26}) and (\ref{2.27}) imply
\begin{equation*}
\int_R\frac{1}{r^{2(m-s)}} (\partial_\theta\partial_r^su)^2 \le
C_{s+1}\big(\int_{\partial_v^+R}\sum_{i=0}^{s+2}(\partial_r^iu)^2
+\int_R\sum_{i=0}^{s+1}\frac{(\partial_r^if)^2}{r^{2(m-i-2)^+}}\big).
\end{equation*}
This, together with (\ref{2.27}), yields (\ref{2.25}) for $k=s+1$.
\end{proof}

Now we prove Lemma \ref{Lemma2.41}.

\begin{proof}[Proof of Lemma \ref{Lemma2.41}] We will only estimate the
$L^2$-norms of $D^\alpha u$ for $|\alpha|=m$.
We will first subtract a polynomial of an appropriate degree from
$u$.  By Lemma \ref{Lemma2.40}, if $\kappa^2a(0)$ is small, we may
find a polynomial $P$ of degree $m-1$ such that $P=0$ on
$\partial\mathcal C_\kappa\cap B_1$, any coefficient $c_k$ of
degree $k$, for $k\le m-1$, in $P$ satisfies
\begin{equation*}|c_k|\le C_k\sum_{|\alpha|\le k-2}|D^\alpha
f(0)|,\end{equation*} and
\begin{equation*}\mathcal L\big(u-P\big)=f_{m}\equiv\text{ $(m-3)$-th remainder
of }f- ({\mathcal
L}-(\partial_{yy}+a(0)\partial_{xx}))P,\end{equation*} where $C_k$
is a positive constant depending only on $\delta_m$, $\kappa$, and
the $C^{m-1}$-norms of $a, b_i$ and $c$. Then the Sobolev
embedding theorem yields
\begin{equation}\label{2.41}|c_k|\le C_{m-1}\|f\|_{H^{m-1}(\mathcal C_\kappa\cap B_1)}.
\end{equation}
Note that $f_m\in V^{m-2}(\mathcal
C_\kappa\cap B_1)$ and $\xi(u-P)\in H^1_0(B_1)$ for any cutoff
function $\xi$ in $B_1$. By Lemma \ref{Lemma1.2}, if
$\kappa^2a(0)$ is small, then $u-P\in V^m(\mathcal C_\kappa\cap
B_1)$. By Lemma \ref{Lemma2.21}, if (\ref{2.19}) holds, then
\begin{align*}
&\int_R\big(\sum_{i=0}^m\frac{(\partial_r^i(u-P))^2}{r^{2(m-i)}}
+\sum_{i=0}^{m-1}\frac{(\partial_\theta\partial_r^i(u-P))^2}{r^{2(m-i)}}\big)\\
\le& C_m\int_{\partial_v^+R}\sum_{i=0}^{m+1}(\partial_r^i(u-P))^2
+\int_R\sum_{i=0}^{m}\frac{(\partial_r^if_m)^2}{r^{2(m-i-2)^+}}.
\end{align*}
Now we apply Lemma \ref{Lemma2.11} to $\mathcal{L}(u-P)=f_m$ to get
\begin{equation*}
\int_R\sum_{|\alpha|=m}|D^\alpha(u-P)|^2 \le
C_m\big(\int_{\partial_v^+R}\sum_{i=0}^{m+1}(\partial_r^i(u-P))^2
+\int_R\sum_{i=0}^{m}\sum_{j=0}^{m-i-2}
\frac{(\partial_r^i\partial_\theta^jf_m)^2}{r^{2(m-i-2)^+}}\big).
\end{equation*}
Then we obtain by (\ref{2.41})
$$\sum_{|\alpha|=m}\|D^\alpha u\|_{L^2(\mathcal C_\kappa\cap \{x<\frac12\})}\le
C_m\big(\sum_{i=0}^{m+1}\|D^iu\|_{L^2(\mathcal C_\kappa\cap
\{x<\frac12\})}+\|f\|_{H^m(\mathcal C_\kappa\cap B_1)}\big).$$
This ends the proof.
\end{proof}

Now, we are ready to prove Theorem \ref{Theorem4.1}.

\begin{proof}[Proof of Theorem \ref{Theorem4.1}] We first note
that $c\le 0$ in $\Omega_\kappa$. For any $\delta>0$, we consider
\begin{equation}\label{3.51}
\mathcal{L}_\delta u\equiv
u_{yy}+(K+\delta)u_{xx}+b_1u_x+b_2u_y+cu=f\quad\text{in
}\Omega_\kappa. \end{equation} This is a uniformly elliptic
differential equation in $\bar{\Omega}_\kappa$. Hence there exists
a solution $u_\delta\in H^1_0(\Omega_\kappa)$. By the classical
theory of uniform elliptic differential equations, we know $u\in
C^\infty(\bar{\Omega}_\kappa\setminus\{0\})\cap
C(\bar{\Omega}_\kappa)$. In the following, we derive estimates on
$u_\delta$ independent of $\delta$. For brevity, we simply write
$u=u_\delta$.

We first estimate $u$ itself. We claim
\begin{equation}\label{3.51a}|u|_{L^\infty(\Omega_\kappa)}\le
C|f|_{L^\infty(\Omega_\kappa)}.\end{equation} To see this, we set
$$w(y)=e^{\alpha d}-e^{\alpha y},$$
where $d$ is chosen so that $d>y$ for any
$(x,y)\in\bar{\Omega}_\kappa$ and $\alpha>0$ is chosen so that
$\mathcal{L}_\delta w\le -1$. Then (\ref{3.51a}) follows from a
simple comparison of $\pm u$ with $|f|_{L^\infty(\Omega)}w$.

Next, we discuss derivatives of $u$. We note that (\ref{3.51}) is
elliptic in any subset $\Omega'$ of $\bar\Omega_\kappa$ away from
the two rays $\theta=\pm\arctan\kappa$. Then by the standard
$H^m$-estimates for solutions of elliptic differential equations
(e.g., Theorem 8.10 in \cite{Gilbarg-Trudinger1983}), we have for
any  $m\ge 2$
\begin{equation}\label{3.52} \|u\|_{H^m(\Omega')}\le
C_m\big(\|u\|_{L^2(\Omega_\kappa)}+\|f\|_{H^{m-2}(\Omega_\kappa)}\big),
\end{equation}
where $C_m$ is a positive constant depending on the distance
between $\partial\Omega'$ and the two rays
$\theta=\pm\arctan\kappa$, the ellipticity constant in $\Omega'$
and the $C^{m-2}$-norms of $K, b_i$ and $c$.

Next, we claim for any $p\in
\partial\Omega_\kappa\cap\partial\mathcal C_\kappa$ and any $m\ge 1$,
there exists a neighborhood $U$ of $p$ such that
\begin{equation}\label{3.53} \|u\|_{H^m(U\cap\Omega_\kappa)}\le
C_m\big(\|u\|_{L^2(\Omega_\kappa)}+\|f\|_{H^{m}(\Omega_\kappa)}\big),\end{equation}
where $C_m$ is a positive constant depending on the distance between
$U$ and the origin, and the $C^m$-norms of $K, b_i$ and $c$. To see
this, we introduce a transform which takes $p$ to the origin, the
ray $\theta=\arctan\kappa$ or $\theta=-\arctan\kappa$ to the
$x$-axis, and a neighborhood of $p$ in $\Omega_\kappa$ to
$D_{\varepsilon}=(-1,1)\times(0,\varepsilon)$. By Corollary
\ref{Cor3.1}, for any cutoff function $\varphi=\varphi(x)$ in
$(-1,1)$ and any $m\ge 0$, there holds
$$\|\varphi u\|_{H^m(D_\varepsilon)}\le
C_m\big(\sum_{k=0}^{m+1}\|
D^ku\|_{L^2(\partial_h^+D_\varepsilon)}+\|
u\|_{L^2(D_\varepsilon)}+\|f\|_{H^m(D_\varepsilon)}\big),$$ as long
as $\varepsilon$ is small. In fact, we may apply Corollary
\ref{Cor3.1} in $D_t=(-1,1)\times(0,t)$ for any $t\in
(\varepsilon/2, \varepsilon)$ and then integrate with respect to $t$
in $(\varepsilon/2, \varepsilon)$. Then we get
$$\|\varphi u\|_{H^m(D_{\varepsilon/2})}\le
C_m\big(\sum_{k=0}^{m+1}\|D^ku\|_{L^2(D_\varepsilon\setminus
D_{\varepsilon/2})}+\|
u\|_{L^2(D_\varepsilon)}+\|f\|_{H^m(D_\varepsilon)}\big).$$ The
first term in the right-hand side can be estimated by (\ref{3.52}).
Hence, we get (\ref{3.53}) easily for an appropriate $U$. We should
note that $U$ depends on $m$. It is obvious that $U$ does not
contain the origin.

With (\ref{3.52}) and (\ref{3.53}) and a simple covering, we obtain
for any $r>0$ and \begin{equation}\label{3.53a}
\|u\|_{H^m(\Omega_\kappa\setminus B_r)}\le
C_m(\|u\|_{L^2(\Omega_\kappa)}+\|f\|_{H^{m}(\Omega_\kappa)}),\end{equation}
where $C_m$ depends only on $r$ and the $C^m$-norms of $K$, $b_i$
and $c$. We emphasize that $C_m$ does not depend on $\delta$.

Next, we discuss the regularity of $u$ in $\Omega_\kappa\cap B_r$.
We claim for any integer $m$ there exists an
$\varepsilon=\varepsilon(m)$ such that if $\delta<\varepsilon^4$
there holds
\begin{equation}\label{3.53b}
\|u\|_{H^m(\Omega_{\kappa}\cap\{x<\varepsilon^5/2\})}\le
C_m\big(\sum_{i=0}^{m+1}\|D^iu\|_{L^2(\Omega_{\kappa}\cap\{x=\varepsilon^5/2\}})
+\|f\|_{H^m(\Omega_{\kappa}\cap\{x<\varepsilon^5/2\})}\big),\end{equation}
where $C_m$ is a positive constant depending only on $m$ and the
$C^m$-norms of $K, b_i$ and $c$. To prove this, we set
$$x=\varepsilon^5 s, \quad y=\varepsilon^4 t.$$
Then
$$|y|<\kappa x \ \Leftrightarrow\ |t|<\varepsilon \kappa s.$$
Let $v(s,t)=u(\varepsilon^5s, \varepsilon^4t)$. Then $v$ satisfies
\begin{equation*}
v_{tt}+a_{\delta}v_{ss}+\varepsilon
^3b_1v_s+\varepsilon^4b_2v_t+\varepsilon^8cv=\varepsilon^8f\quad\text{in
}\Omega_{\varepsilon\kappa}\cap B_1,
\end{equation*}
where $a_{\delta}=(K+\delta)\varepsilon^{-2}$ and $K$, $b_i$ and $c$
are evaluated at $(\varepsilon^5s, \varepsilon^4t)$. Note $K(0,0)=0$
and hence for $(s,t)\in \Omega_{\varepsilon\kappa}\cap B_1$
$$K(\varepsilon^5s,
\varepsilon^4t)\le
|DK|_{L^\infty}\sqrt{\varepsilon^{10}s^2+\varepsilon^8t^2}\le
\varepsilon^4|DK|_{L^\infty}.$$ This implies
$$(\varepsilon \kappa)^2a_{\delta}(0,0)=\kappa^2\delta,$$ and
$$
a_\delta\le
\varepsilon^2|DK|_{L^\infty}+\frac{\delta}{\varepsilon^2}\text{ in
}\Omega_{\varepsilon\kappa}\cap B_1.$$ Now we take $\varepsilon$
small so that $\kappa\varepsilon^2\le\delta_m$ and
$\kappa\varepsilon\le \kappa_m$, where $\delta_m$ and $\kappa_m$
are as in Lemma \ref{Lemma2.41}. Then if $\delta<\varepsilon^4$,
Lemma \ref{Lemma2.41} implies $v\in
H^k(\Omega_{\varepsilon\kappa})$ and
\begin{equation}\label{3.55}
\|v\|_{H^m(\Omega_{\varepsilon\kappa}\cap\{s<1/2\})}\le
C_m\big(\sum_{i=0}^{m+1}\|D^iv\|_{L^2(\Omega_{\varepsilon\kappa}\cap\{s=1/2\}})
+\|f\|_{H^m(\Omega_{\varepsilon\kappa}\cap\{s<1/2\})}\big),\end{equation}
where $C_m$ is a positive constant depending only on $m$ and the
$C^m$-norms of $K, b_i$ and $c$. Obviously, (\ref{3.55}) implies
(\ref{3.53b}). With a similar trick, we then get
\begin{equation}\label{3.53c}
\|u\|_{H^m(\Omega_{\kappa}\cap\{x<\varepsilon^5/2\})}\le
C_m\big(\sum_{i=0}^{m+1}
\|D^iu\|_{L^2(\Omega_{\kappa}\cap\{\varepsilon^5/2<x<\varepsilon^5\}})
+\|f\|_{H^m(\Omega_{\kappa}\cap\{x<\varepsilon^5\})}\big).\end{equation}

With (\ref{3.53a}) and (\ref{3.53c}), we conclude the following
result: For any integer $m$ there exists an
$\varepsilon=\varepsilon(m)$ such that the solution $u_\delta$ of
(\ref{3.51}) with $u_\delta=0$ on $\partial\Omega_\kappa$ for
$\delta<\varepsilon^4$ satisfies
\begin{equation*}
\|u_\delta\|_{H^m(\Omega_{\kappa})}\le
C_m(\|u\|_{L^2(\Omega_\kappa)}+\|f\|_{H^{m+1}(\Omega_{\kappa})}),\end{equation*}
where $C_m$ is a positive constant depending only on $m$ and the
$C^m$-norms of $K, b_i$ and $c$. With (\ref{3.51a}) and the Sobolev
embedding theorem, we obtain for any $m\ge 1$
\begin{equation*}
\|u_\delta\|_{H^m(\Omega_{\kappa})}\le
C_m\|f\|_{H^{m+1}(\Omega_{\kappa})}.\end{equation*} It is easy to
get a sequence of $\delta\to0$ and a $u\in \cap_{m=1}^\infty
H^m(\Omega_\kappa)\cap H^1_0(\Omega_\kappa)$ such that
$$u_\delta\to u\quad\text{in }H^m(\Omega_\kappa) \text{ for any
}m\text{ as }\delta\to0.$$ Therefore, $u$ is a solution of
(\ref{4.1}) and satisfies $u=0$ on $\partial\Omega_\kappa$ and
(\ref{4.4}).
\end{proof}

\begin{rmrk}\label{Remark-forTheorem4.1}
It is clear that Theorem \ref{Theorem4.1} still holds
if $\Omega_\kappa$ is replaced by $\Omega_{\kappa_1, \kappa_2}$ with
the property that $\partial\Omega_{\kappa_1,
\kappa_2}\setminus\{0\}$ is smooth and that in a small neighborhood
of the origin $\partial\Omega_{\kappa_1,\kappa_2}$ is given by
smooth functions $y=\kappa_1(x)$ and $y=\kappa_2(x)$ over a small
interval $[0, d]$ with $\kappa_1(0)=\kappa_2(0)=0$ and
$\kappa_1'(0)>0>\kappa_2'(0)$. To see this, we simply note that
there exists a smooth transform in $\bar{\Omega}_{\kappa_1,
\kappa_2}$ such that $F(\Omega_{\kappa_1, \kappa_2}\cap U)=\mathcal
C_{\kappa}\cap V$ for a positive constant $\kappa$ and neighborhoods
$U$ and $V$ of the origin.
\end{rmrk}

\begin{rmrk}\label{Remark2-forTheorem4.1} We also note that $c\le 0$ can
be replaced by $c\le\varepsilon$ for $\varepsilon>0$ sufficiently small.
This is standard for elliptic differential equations.
\end{rmrk}

\section{The Cauchy Problem in Non-smooth Hyperbolic Regions}
\label{Section-HyperbolicExistence}

In this section, we will discuss Cauchy problems for hyperbolic
equations in $\mathbb R^2$ when the initial curve has an angular
point. We will discuss uniformly hyperbolic equations here and treat
degenerate hyperbolic equations in the next section.

It is well known that the Cauchy problem for linear hyperbolic
differential equation is well-posed
in a domain whose boundary is a smooth non-characteristic curve. A
standard example of such a domain is the upper half plane. However,
we cannot apply directly results for smooth domains to non-smooth
domains. In  this
section, we will prove by hand the existence of solutions of
Cauchy problems for hyperbolic
equations if the initial curve is not smooth and has an angular point.
The
method is based on energy estimates and is particularly designed for
non-flat domains. The regularity of these solutions depends
essentially on a class of {\it compatibility conditions} of Cauchy
data and nonhomogeneous terms at angular points.

Throughout this section, we fix a function $y=\kappa(x)$ on
$\mathbb R$ with $\kappa(0)=0$ satisfying
$$\text{$y=\kappa(x)$ is
Lipschitz in $\mathbb R$ and smooth for any $x\neq 0$,}$$ and
$$\text{$y=\kappa(x)$ is strictly decreasing for $x<0$ and strictly
increasing for $x>0$.}$$ Hence for any $\tau>0$, $\kappa(x)=\tau$
has two roots, one positive and one negative. An important example
of such a function is given by $y=\kappa|x|$ for a positive
constant $\kappa$.

For a fixed positive constant  $y_0$, we set
$$\Omega_{\kappa, y_0}=\{(x,y); \kappa(x)<y<y_0\}.$$
For brevity, we simply write $\Omega$ instead of $\Omega_{\kappa,
y_0}$. We denote by $\partial_b\Omega$ and $\partial_t\Omega$ the
bottom and top boundaries of $\Omega$, i.e.,
$$\partial_b\Omega=\{(x,y);\ y=\kappa(x)<y_0\},\quad\partial_t\Omega=\{(x,y);\ \kappa(x)<{y_0},
y=y_0\}.$$ In the following, we consider
\begin{equation}\label{c1}Lu\equiv u_{yy}-(au_x)_x+b_1u_x+b_2u_y+cu=f\quad\text{in
}\Omega,\end{equation} where $a, b_1, b_2$ and $c$ are smooth
functions in $\Omega$ satisfying
\begin{equation}\label{c2} a\ge a_0\quad\text{in
}\Omega,\end{equation} for a positive constant $a_0$. Obviously,
$y=\kappa(x)$ is space-like if
\begin{equation}\label{c2a} a\kappa_x^2\le\eta_0\quad\text{on
}\partial_b\Omega,\end{equation} for a constant $\eta_0\in(0,1)$.
Our goal is to prove that the Cauchy problem of (\ref{c1}) in
$\Omega$ is well-posed for Cauchy data prescribed on
$\partial_b\Omega$.
We point out that
$\partial_b\Omega$, as an initial curve, is not smooth and
has an angular point.

For any nonnegative integers $m\ge l$, we define $H^{(m,l)}(\Omega)$
($H_{0b}^{(m,l)}(\Omega)$) to be the closure of all
$C^{\infty}(\Omega)$ functions (which vanish to all orders at
$\partial_b\Omega$), in the norm
\begin{equation*}
\parallel u\parallel^{2}_{(m,l)}=\int_{\Omega}
\sum_{j=0}^l\sum_{i=0}^{m-j}(\partial_{x}^{i}\partial_y^ju)^{2}.
\end{equation*}
Obviously, the usual Sobolev space $H^m(\Omega)$ is a subset of
$H^{(m,l)}(\Omega)$.  The $L^{2}(\Omega)$ inner product will as
usual be denoted by $(\cdot,\cdot)$. A simple calculation shows that
the formal adjoint $L^*$ of $L$ is given by
$$L^*u=u_{yy}-(au_x)_x-(b_1u)_x-(b_2u)_y+cu.$$

It is convenient to first establish an existence result for
(\ref{c1}) with homogeneous Cauchy data and with $f$ vanishing to
high order on $\partial_b\Omega$.

\begin{lemma}\label{Lemma-weak} Let $m$ be a positive integer, $a\in C^{m+1}(\bar\Omega)$
and $b_1, b_2, c\in C^m(\bar\Omega)$ satisfying (\ref{c2}) and
(\ref{c2a}). Then for any $f\in H_{0b}^{(m,0)}(\Omega)$, there
exists a $u\in H_{0b}^{(m+1,1)}(\Omega)$ satisfying
\begin{equation}\label{c3}
(u ,L^{*}v)=(f,v)\quad\text{for any $v\in C^{\infty}(\Omega)$ with
$v=v_{y}=0$ on $\partial_t\Omega$}.
\end{equation}
\end{lemma}

We note that (\ref{c2a}) holds automatically for arbitrary
nonnegative $a$ if $\partial_b\Omega$ is a horizontal line, i.e.,
$\kappa\equiv0$. It is clear from the proof that $\eta_0$ in
(\ref{c2a}) is allowed to be 1 in Lemma \ref{Lemma-weak}.

\begin{proof} Let $\widehat{C}^{\infty}(\Omega)$ consist of all
$C^{\infty}(\Omega)$ functions $v$ with $v=v_{y}=0$ on
$\partial_t\Omega$. We consider a fixed $v\in
\widehat{C}^\infty(\Omega)$. For a large constant $\lambda$ to be
determined, we consider
\begin{align*}
\sum_{s=0}^{m}(-1)^{s}\lambda^{-s}\partial^{2s}_x\varphi&=v\quad\text{in
}\Omega,\\
\varphi=\cdots=\partial_x^{m-1}\varphi&=0\quad\text{on
}\partial_b\Omega.\end{align*} This is a boundary value problem
related to an ODE for each $y\in(0, y_0)$,  and therefore the theory
of such equations guarantees the existence of a unique solution
$\varphi\in C^{2m}(\Omega)$. Set
$$w(x,y)=\int_{\kappa|x|}^y e^{\lambda t}\varphi(x,t)dt.$$ Then
$w$ satisfies
\begin{align}\label{c6}\begin{split}
\sum_{s=0}^{m}(-1)^{s}\lambda^{-s}\partial^{2s}_x(e^{-\lambda
y}w_y)&=v\quad\text{in
}\Omega,\\
w=w_y=\partial_xw_y=\cdots=\partial_x^{m-1}w_y&=0\quad\text{on
}\partial_b\Omega.\end{split}\end{align} We note that $w$
satisfies extra boundary conditions
\begin{equation}\label{c7}
\partial_{x}^{i}\partial_{y}^{j}w|_{\partial_b\Omega}=0,\quad\text{for }i+j\leq m.
\end{equation}
To see this, we simply differentiate $w=0$ along
$\partial_b\Omega$ to get
\begin{equation}\label{c7a}n_2w_x-n_1w_y=0\quad\text{on
}\partial_b\Omega,\end{equation} where $(n_1, n_2)$ is the outward
unit normal vector of $\partial_b\Omega$. With $w_y=0$ on
$\partial_b\Omega$, we get easily $w_x=0$ on $\partial_b\Omega$. A
simple induction argument then yields (\ref{c7}). We note that
$$(n_1, n_2)=\frac1{\sqrt{1+\kappa_x^2}}(\kappa_x, -1).$$

By taking $\lambda$ sufficiently large, we claim
\begin{equation}\label{c8}
\big(Lw,\sum_{s=0}^{m}(-1)^{s}\lambda^{-s}\partial^{2s}_x(e^{-\lambda
y}w_y)\big) \geq C\| w\|_{(m+1,1)}^{2},
\end{equation}
where $C$ is a positive constant depending on $m$, $a_0$, the
$C^{m+1}$-norm of $a$ and the $C^m$-norms of $b_1, b_2$ and $c$.
To prove this, we integrate by parts each term in the left hand
side of (\ref{c8}) repeatedly with the help of (\ref{c7}). First
for $1\le s\le m$, we have
\begin{align*}
&\big(w_{yy}-(aw_x)_x,(-1)^{s}e^{-\lambda
y}\partial_{x}^{2s}w_y\big)\\
 = &\int_{\Omega}e^{-\lambda
y}\partial_x^s\partial_y^2w\partial_x^s\partial_yw+\int_{\Omega}e^{-\lambda
y}\partial^{s+1}_{x}w_y\partial_{x}^{s}(aw_x)
-\delta_{sm}\int_{\partial\Omega} e^{-\lambda
y}\partial_{x}^{s-1}\partial_{yy}w\partial_{x}^{s}\partial_ywn_{1}\\
=&\int_{\Omega} e^{-\lambda
y}\big(\frac{1}{2}\lambda(\partial_{x}^{s}\partial_yw)^{2}+\frac{1}{2}(\lambda
a-a_y)(\partial_x^{s+1}w)^2-\partial_x^s\partial_yw\partial_x(
\sum_{i=1}^{s}C_s^i\partial_{x}^{i}a
\partial_{x}^{s+1-i}w)\big)\\
& \quad+\int_{\partial\Omega} e^{-\lambda
y}\big(\frac{1}{2}a(\partial_x^{s+1}w)^2n_2+\frac12(\partial_x^s\partial_yw)^2n_{2}-\delta_{sm}
\partial_{x}^{s-1}\partial_{yy}w\partial_{x}^{s}\partial_ywn_{1}\big),\end{align*}
and
\begin{align*}
&\big(b_1w_x+b_2w_2+cw,(-1)^{s}e^{-\lambda
y}\partial_{x}^{2s}w_y\big)\\
=&\int_\Omega e^{-\lambda
y}\partial_x^s(b_1w_x+b_2w_y+cw)\partial_x^s\partial_yw.
\end{align*}
For $s=0$, we simply have $$(Lw, e^{-\lambda y}w)=\int_\Omega
e^{-\lambda y}\big(\frac12w_y^2+\frac12(\lambda
a-a_y)w_x^2+(b_1w_x+b_2w_y+cw)w_y\big).$$ By taking $\lambda$
large enough and using a Poincar\'{e} type inequality to estimate
$\|w\|_{L^{2}(\Omega)}$, we obtain
\begin{align*}
&\big(Lw,\sum_{s=0}^{m}(-1)^{s}\lambda^{-s}e^{-\lambda
y}\partial_{x}^{2s}w_y\big) \ge C\|w\|_
{(m+1,1)}^{2}\\
&+\lambda^{-m}\int_{\partial\Omega} e^{-\lambda
y}\big(\frac{1}{2}a(\partial_x^{m+1}w)^2n_2
+\frac12(\partial_x^m\partial_yw)^2n_{2}-
\partial_{x}^{m-1}\partial_{yy}w\partial_{x}^{m}\partial_ywn_{1}\big),
\end{align*}
where $\lambda$ and $C$ are positive constants depending on $m$,
$a_0$, the $C^{m+1}$-norm of $a$ and the $C^m$-norms of $b_1, b_2$
and $c$. Note that the boundary integral is nonnegative  on
$\partial_t\Omega$.  We now study the boundary integral on
$\partial_b\Omega$. We first note
$\partial^{m-1}_x\partial_yw=\partial^m_xw=0$ on $\partial_b\Omega$
by (\ref{c7}). Then by an argument as similar as in proving
(\ref{c7a}), we have
\begin{equation*}
n_1\partial_{x}^{m-1}\partial_{yy}w-n_{2}\partial_{x}^{m}\partial_yw=0,\quad
n_{1}\partial_{x}^{m}\partial_{y}w-n_{2}\partial_{x}^{m+1}w=0\quad\text{on
}\partial_b\Omega.
\end{equation*}
It follows that the boundary integral on $\partial_b\Omega$ is given
by \begin{align*}&\int_{\partial\Omega} e^{-\lambda
y}\big(\frac{1}{2}a(\partial_x^{m+1}w)^2n_2
+\frac12(\partial_x^m\partial_yw)^2n_{2}-
\partial_{x}^{m-1}\partial_{yy}w\partial_{x}^{m}\partial_ywn_{1}\big)\\
=&\frac12\int_{\partial\Omega}e^{-\lambda
y}\big(a\frac{n_1^2}{n_2^2}-1\big)(\partial_x^{m}\partial_yw)^2n_2=
\frac12\int_{\partial\Omega}e^{-\lambda
y}\big(a\kappa_x^2-1\big)(\partial_x^{m}\partial_yw)^2n_2.\end{align*}
This is nonnegative by (\ref{c2a}) and $n_2<0$ on
$\partial_b\Omega$. Then (\ref{c8}) holds.

Next we claim
\begin{equation}\label{c9}
\| v\|_{(-m,0)} \leq C\|w\|_{(m+1,1)}.
\end{equation} Here
$\|\cdot\|_{(-m,0)}$ is the norm on the dual space
$H_{0b}^{(-m,0)}(\Omega)$ of $H_{0b}^{(m,0)}(\Omega)$.  This dual
space may be obtained as the completion of $L^{2}(\Omega)$ in the
norm $\|\cdot\|_{(-m,0)}$. To get (\ref{c9}), we simply note
\begin{align*}
&\| v\|_{(-m,0)}=\sup_{z\in H_{0b}^{(m,0)}(\Omega)}\frac{|(v,z)|}{\|
z\|_{(m,0)}}\\
=&\sup_{z\in
H_{0b}^{(m,0)}(\Omega)}\frac{|(\sum_{s=0}^{m}(-1)^{s}\lambda^{-s}
\partial^{2s}_x(e^{-\lambda
y}w_y), z)|}{\| z\|_{(m,0)}} \leq C\|w\|_{(m+1,1)}.
\end{align*}

Now, a simple integration by parts yields
$$(w,L^*v)=(Lw, v)\quad\text{for any }v\in \widehat{C}^{\infty}(\Omega).$$
By (\ref{c8}), we obtain
\begin{align*}
&\|w\|_{(m+1,1)}\| L^{*}v \|_{(-m-1,-1)}\ge
(w,L^{*}v)=(Lw,v)\\
=&\big(Lw,\sum_{s=0}^{m}(-1)^{s}\lambda^{-s}\partial^{2s}_x(e^{-\lambda
y}w_y)\big)\geq C\|w\|_{(m+1,1)}^{2},
\end{align*}
and hence with (\ref{c9})
\begin{equation}\label{c10}
\|v\|_{(-m,0)}\leq C \| L^*v\|_{(-m-1,-1)}\quad\text{for any
}v\in\widehat{C}^{\infty}(\Omega).
\end{equation}
Consider the linear functional $F:
L^{*}\widehat{C}^{\infty}(\Omega)\rightarrow\mathbb{R}$ given by
\begin{equation*}
F(L^{*}v)=(f,v).
\end{equation*}
By (\ref{c10}), we have
\begin{equation*}
|F(L^{*}v)|\leq\| f\|_{(m,0)}\| v\|_{(-m,0)}\leq C\| f\|_{(m,0)} \|
L^{*}v\|_{(-m-1,-1)}.
\end{equation*}
Hence $F$ is a bounded linear functional on the subspace
$L^{*}\widehat{C}^{\infty}(\Omega)$ of $H_{0b}^{(-m-1,-1)}(\Omega)$.
Thus we can apply the Hahn-Banach Theorem to obtain a bounded
extension of $F$ defined on $H_{0b}^{(-m-1,-1)}(\Omega)$ such that
$\|F\|\le C\| f\|_{(m,0)}$. It follows that there exists a $u\in
H^{(m+1,1)}_{0b}(\Omega)$ such that
\begin{equation*}
F(z)=(u,z)\quad\text{for any }z\in H_{0b}^{(-m-1,-1)}(\Omega).
\end{equation*}
Now restrict $z$ back to $L^{*}\widehat{C}^{\infty}(\Omega)$ to
obtain (\ref{c3}).
\end{proof}

Next, we discuss the regularity of solutions in Lemma
\ref{Lemma-weak} in usual Sobolev spaces.  The Sobolev space of
square integrable derivatives up to and including order $m$ will be
denoted by $H^{m}(\Omega)$ with norm $\|\cdot\|_{m}$, and the
completion of $C^{\infty}(\Omega)$ functions which vanish to all
order at $\partial_b\Omega$ in the norm $\|\cdot\|_{m}$ will be
denoted by $H_{0b}^{m}(\Omega)$.

\begin{cor}\label{Cor-weak1} Under the hypotheses of
Lemma \ref{Lemma-weak}, if $f\in H_{0b}^{m}(\Omega)$, there exists a
unique solution $u\in H_{0b}^{m+1}(\Omega)$ of (\ref{c1}).\end{cor}

\begin{proof} Obviously,
$H_{0b}^{m}(\Omega)\subset H_{0b}^{(m,0)}(\Omega)$. Let $u\in
H_{0b}^{(m+1,1)}(\Omega)$ be the function given in Lemma
\ref{Lemma-weak} so that (\ref{c3}) holds.

We first consider $m=1$. By $u\in H_{0b}^{(2,1)}(\Omega)$, we have
$u, u_x\in H^{1}(\Omega)$ with $u=u_x=0$ on $\partial_b\Omega$ in
the $L^2(\partial_b\Omega)$ sense. We integrate by parts to obtain
\begin{equation}\label{c16}
-(u_y+b_2u, v_y)= \big(f+(au_x)_x-b_1u_x+(b_{2,y}-c)u,
v\big)\quad\text{for any
}v\in\widehat{C}^{\infty}(\Omega).\end{equation} A standard argument
using difference quotients in the $y$-direction
implies $(u_y+b_2u)_y\in L^2_{\text{loc}}(\Omega)$ and
$$(u_y+b_2u)_y=f+(au_x)_x-b_1u_x+(b_{2,y}-c)u.$$
Then $u_{yy}\in L^2_{\text{loc}}(\Omega)$ and
$$u_{yy}-(au_x)_x+b_1u_x+b_2u_y+cu=f\quad\text{in }\Omega.$$
This implies easily that $u_{yy}\in L^2(\Omega)$ and hence $u\in
H^2(\Omega)$. An integration by parts of (\ref{c16}) then yields
$u_y=0$ on $\partial_b\Omega$ in the $L^2(\partial_b\Omega)$ sense.
Last, by $u=|\nabla u|=0$ on $\partial_b\Omega=0$, (\ref{c8}) with
$m=1$ yields
\begin{equation*}
\big(Lu,\sum_{s=0}^{1}(-1)^{s}\lambda^{-s}\partial^{2s}_x(e^{-\lambda
y}u_y)\big) \geq C\| u\|_{(2,1)}^{2},
\end{equation*}
from which the uniqueness follows.

Now we assume $m\ge 2$. We already proved that $u\in H^2(\Omega)$
and that (\ref{c3}) holds. We need to prove
$$\partial_x^i\partial_y^ju\in L^2(\Omega)\quad\text{for any }i+j\le
m+1,$$ and
$$\partial_x^i\partial_y^ju|_{\partial_b\Omega}=0\quad\text{for any }i+j\le
m.$$ This follows easily from (\ref{c1}), $u\in
H_{0b}^{(m+1,1)}(\Omega)$  and $f\in H_{0b}^{m}(\Omega)$.
\end{proof}

Corollary \ref{Cor-weak1} yields the existence of a regular solution
of (\ref{c1}) for homogeneous Cauchy data and $f$ vanishing to high
order on $\partial_b\Omega$.  However, our main concern is to solve
(\ref{c1}) for general $f$ and Cauchy data
\begin{equation}\label{c21}u=\varphi,\quad u_y=\psi\quad\text{on
}\partial_b\Omega.\end{equation} Since $\partial_b\Omega$ has an
angular point at the origin, there is a natural compatibility
condition which we will derive next. As $\partial_b\Omega$ is the
graph given by $y=\kappa(x)$ over $\mathbb R$, we may assume
$\varphi$ and $\psi$ are functions of $x\in \mathbb R$.

\begin{lemma}\label{Lemma-extension0} Let $m\ge 2$ be an integer
and  $\varphi\in C(\partial_b\Omega)\cap
C^m(\partial_b\Omega\setminus\{0\})$, $\psi\in
C(\partial_b\Omega)\cap C^{m-1}(\partial_b\Omega\setminus\{0\})$
and $f\in C^{m-2}(\bar\Omega)$. Suppose  (\ref{c2a}) is satisfied.
Then there exists a $u\in C^m(\bar\Omega)$ such that
\begin{equation}\label{c25} u=\varphi,\ u_y=\psi,\text{ and
}\partial^\alpha(Lu-f)=0\quad\text{ on
}\partial_b\Omega,\end{equation} for any $|\alpha|\le m-2$ if and
only if there hold compatibility conditions $\mathcal
C_{i}(\varphi,\psi,f)$ for $i=1,\cdots, m$.
\end{lemma}

The compatibility condition $\mathcal C_{i}(\varphi,\psi,f)$  is
imposed on (one-sided) derivatives of $\varphi,\psi,f$ and $\kappa$
up to order $i$ at the origin. The formulation of such a condition
will be given in the proof below, from which it is clear that
$\mathcal C_m(\varphi,\psi,f)$ makes sense for $\varphi\in
C(\partial_b\Omega)\cap C^m(\partial_b\Omega\setminus\{0\})$,
$\psi\in C(\partial_b\Omega)\cap
C^{m-1}(\partial_b\Omega\setminus\{0\})$ and $f\in C^{m-2}(\bar
\Omega)$.

\begin{proof} First, we assume there exists a function $u\in C^1(\bar
\Omega)$ satisfying (\ref{c21}). Then a simple differentiation
yields
$$u_x+\kappa_x u_y=\varphi_x\quad\text{on
}\partial_b\Omega\setminus\{0\},$$ or
$$u_x=\varphi_x-\kappa_x \psi\quad\text{on
}\partial_b\Omega\setminus\{0\}.$$ Letting $x\to0+$ and $x\to0-$,
we have a  compatibility condition
$$\varphi_x(0+)-\kappa_x(0+)\psi(0)=\varphi_x(0-)-\kappa_x(0-)\psi(0),$$
or
\begin{equation}\label{c22}
\psi(0)\big(\kappa_x(0+)-\kappa_x(0-)\big)=\varphi_x(0+)-\varphi_x(0-).\end{equation}
If (\ref{c22}) holds, then
$$u_x(0)=-\frac{\kappa_x(0+)\varphi_x(0-)+\kappa_x(0-)\varphi_x(0+)}
{\kappa_x(0+)-\kappa_x(0-)}.$$ It is easy to check that for any
$\varphi\in C(\partial_b\Omega)\cap
C^1(\partial_b\Omega\setminus\{0\})$ and $\psi\in
C(\partial_b\Omega)$ satisfying (\ref{c22}), there exists a $u\in
C^1(\bar\Omega)$ satisfying (\ref{c21}). We denote by $\mathcal
C_{1}(\varphi,\psi,f)$ the compatibility condition (\ref{c22}),
which in fact is independent of $f$.

The discussion for higher order derivatives is more complicated.
For an integer $m\ge 2$, we assume we already derived $\mathcal
C_{i}(\varphi,\psi,f)$ for $i=1,\cdots, m-1$. Now let $u\in
C^m(\bar \Omega)$ satisfy (\ref{c25}). For any multi-index
$\alpha\in \mathbb Z^2_+$ with $|\alpha|=m-2$, a simple
calculation yields
$$\partial_x^{\alpha_1}\partial_y^{\alpha_2}u(p)=
\begin{cases}a^{\frac{\alpha_2}2}(p)\partial_x^mu(p)+\cdots
&\text{ if $\alpha_2$ is even},\\
a^{\frac{\alpha_2-1}2}(p)\partial_x^{m-1}\partial_yu(p)+\cdots
&\text{ if $\alpha_2$ is odd}, \end{cases}$$ where $\cdots$ denotes
a linear combination of derivatives of $u$ at $p$ with order $\le
m-1$ and derivatives of $f$ at $p$ with order $\le m-2$. Now we
apply $\partial_x^m$ to $u=\varphi$ and $\partial_x^{m-1}$ to
$u_y=\psi$ and  evaluate at $p\in\partial_b\Omega\setminus\{0\}$.
Then we get on $\partial_b\Omega\setminus\{0\}$
\begin{align*}
\sum_{i=0}^mC_m^i\kappa_x^i\partial_x^{m-i}\partial^i_yu|_p&=\varphi^{(m)}+\cdots,\\
\sum_{i=0}^{m-1}C_{m-1}^i\kappa_x^i\partial_x^{m-1-i}\partial^{i+1}_yu|_p&=
\psi^{(m-1)}+\cdots,\end{align*} where $\cdots$ denotes derivatives
of $u$ at $p$ with order $\le m-1$. By a simple substitution of
$\partial^\alpha u(p)$ with $\alpha_2\ge 2$, we obtain at
$p\in\partial_b\Omega\setminus\{0\}$
\begin{align*} \big(\sum_{0\le 2i\le
m}C_m^{2i}\kappa_x^{2i}a^i\big)\partial_x^mu+\big(\sum_{0\le 2i+1\le
m}C_m^{2i+1}\kappa_x^{2i+1}a^i\big)\partial_x^{m-1}\partial_yu&=
\varphi^{(m)}+\cdots,\\
\big(\sum_{0\le 2i+1\le
m-1}C_{m-1}^{2i+1}\kappa_x^{2i+1}a^{i+1}\big)\partial_x^mu+\big(\sum_{0\le
2i\le m-1}C_{m-1}^{2i}\kappa_x^{2i}a^i)\partial_x^{m-1}\partial_yu&=
\psi^{(m-1)}+\cdots,\end{align*} where $\cdots$ denotes a linear
combination of derivatives of $u$ at $p$ with order $\le m-1$ and
derivatives of $f$ at $p$ with order $\le m-2$. This is a $2\times
2$ linear system for $\partial_x^mu(p)$ and
$\partial_x^{m-1}\partial_y u(p)$. A straightforward calculation
shows that the determinate of the coefficient matrix is given by
$$(1-a\kappa_x^2)^{m-1}|_p,$$
which is nonzero by (\ref{c2a}). This implies that $\partial_x^mu$
and $\partial_x^{m-1}\partial_yu$, and hence all other $m$-th
order derivatives of $u$, at $p\in
\partial_b\Omega\setminus\{0\}$, can be expressed as a linear
combination of $\varphi^{(m)}(p)$, $\psi^{(m-1)}(p)$, derivatives
of $u$ at $p$ with order $\le m-1$ and derivatives of $f$ at $p$
with order $\le m-2$. Now we consider $p=0$. In this case, there
are four linear equations for $\partial_x^mu(0)$ and
$\partial_x^{m-1}\partial_yu(0)$ arising from $x\to0+$ and
$x\to0-$. This implies that there are two compatibility conditions
similar to (\ref{c22}) involving $\varphi^{(m)}(0+)$,
$\varphi^{(m)}(0-)$, $\psi^{(m-1)}(0+)$, $\psi^{(m-1)}(0-)$,
$\kappa^{(i)}(0+)$ and $\kappa^{(i)}(0-)$, $i=1,\cdots, m$. We
denote by $\mathcal C_m(\varphi,\psi,f)$ this compatibility
condition. If $\mathcal C_m(\varphi,\psi,f)$ is satisfied, then
$\partial_x^mu(0)$ and $\partial_x^{m-1}\partial_yu(0)$, and hence
all other $m$-th order derivatives of $u$ at $0$, can be expressed
as a linear combination of $\varphi^{(m)}(0+)$,
$\varphi^{(m)}(0-)$, $\psi^{(m-1)}(0+)$, $\psi^{(m-1)}(0-)$,
derivatives of $u$ at $0$ with order $\le m-1$ and derivatives of
$f$ at $0$ with order $\le m-2$. \end{proof}



Now we are ready to solve the Cauchy problem (\ref{c1}) and
(\ref{c21}).

\begin{thrm}\label{Thrm-existence} Let $m\ge 2$ be an integer
and $\varphi\in H^{m+1}(\partial_b\Omega)$, $\psi\in
H^m(\partial_b\Omega)$ and $f\in H^m(\Omega)$. Suppose (\ref{c2}),
(\ref{c2a}) and $\mathcal C_{i}(\varphi,\psi,f)$, $i=1,\cdots, m$,
are satisfied. Then the Cauchy problem (\ref{c1}) and (\ref{c21})
admits a unique solution $u\in H^m(\Omega)$. Moreover,
\begin{equation}\label{c30} \|u\|_{m,\Omega}\le C(\|\varphi\|_{m+1, \partial_b\Omega}
+\|\psi\|_{m,\partial_b\Omega}+\|f\|_{m,\Omega}),\end{equation}
where $C$ is a positive constant depending only on $m$, $a_0$,
$\eta_0$, the $C^{m+1}$-norm of $a$ and the $C^m$-norms of
$b_1,b_2$ and $c$. \end{thrm}

Here and thereafter, we denote by $\|\cdot\|_{m,\Omega}$ and
$\|\cdot\|_{m,\partial_b\Omega}$ the $H^m$-norms in $\Omega$ and
$\partial_b\Omega$ respectively.

\begin{proof} By the Sobolev embedding, we have $\varphi\in
C(\partial_b\Omega)\cap C^m(\partial_b\Omega\setminus\{0\})$,
$\psi\in C(\partial_b\Omega)\cap
C^{m-1}(\partial_b\Omega\setminus\{0\})$ and $f\in
C^{m-2}(\Omega)$. Hence the compatibility condition $\mathcal
C_{i}(\varphi,\psi,f)$ makes sense for $i=1,\cdots, m$. By Lemma
\ref{Lemma-extension0}, there exists a $v\in C^m(\bar\Omega)$ such
that $v=\varphi$, $v_y=\psi$ and $\partial^\alpha(f-Lv)=0$  on
$\partial_b\Omega$ for any $|\alpha|\le m-2$. This implies
$f-Lv\in H^{m-1}_{0b}(\Omega)$. By Corollary \ref{Cor-weak1},
there exists a $w\in H^m_{0b}(\Omega)$ such that $Lw=f-Lv$. Then
$u=v+w$ is a required solution.

We note that (\ref{c30}) is the classical energy estimates. The
proof is identical to that for Cauchy problems with the initial
curve as the $x$-axis. For example, the $H^1$-estimate is based on
integrating the product of (\ref{c1}) and $u_y$. We omit details.
\end{proof}

In this paper, we only need the existence part in Theorem
\ref{Thrm-existence}. The estimate (\ref{c30}) depends on the lower
bound $a_0$ of $a$ and is not sufficient for our application.
In the next
section, we will derive an estimate independent of $a_0$ under extra
assumptions on $a$.

\section{A Priori Estimates in the Hyperbolic
Regions}\label{Section-HyperbolicEstimates}

In this section, we will derive estimates of the solutions established
in the previous section which are independent of the hyperbolicity
constant. Such estimates will enable us to establish the existence
of solutions to the Cauchy problem for degenerate hyperbolic
equations when the initial curve has angular points.

Let $y=\kappa(x)$ and $\Omega\subset\mathbb R^2$ be as defined in
the beginning of Section \ref{Section-HyperbolicExistence}.
We consider an  equation  of the following form
\begin{equation}\label{3.4}Lu\equiv u_{yy}-aKu_{xx}
+b_1u_x+b_2u_y+cu=f\quad\text{in }\Omega
\end{equation} with the Cauchy data
\begin{equation}\label{3.4z} u=\varphi, \ u_y=\psi\quad\text{on }\partial_b\Omega,\end{equation}
where $a$, $b_1$, $b_2$, $c$ and $K$ are smooth functions
satisfying
\begin{align}\label{3.5}
&\lambda \le a \le \Lambda \quad \text{in }\Omega, \\
\label{4.7}&0< K\le 1 \quad \text{in }\Omega,\end{align} and
\begin{equation}\label{3.6q}|b_1|\le C_b\big(\sqrt{K}+|K_x|\big) \quad\text{in }\Omega,
\end{equation}
for positive constants $\lambda\le \Lambda$ and $C_b$. We always
assume that $\partial_b\Omega$ is space-like, i.e.,
\begin{equation}\label{3.7}
aK\kappa_x^2\le \eta_0,\end{equation} for a constant $\eta_0\in
(0,1)$. In the following, we also assume
\begin{equation}\label{3.6z}K_x^2\le C_K^2K_y\quad\text{in }\Omega,
\end{equation}
and
\begin{equation}\label{3.8} \big(y-\kappa(x)\big)^d\le
C_KK(x,y)\quad\text{for any }(x,y)\in\Omega,\end{equation} where
$C_K$ is a positive constant and $d$ is a positive integer. Note
that (\ref{3.6z}) implies in particular $K_y\ge 0$.

Here, $K$ is allowed to be zero along $\partial_b\Omega$. If this
happens, (\ref{3.4}) is degenerate there and (\ref{3.7}) holds
automatically. Conditions (\ref{3.6q}), (\ref{3.6z}) and (\ref{3.8})
are introduced to overcome the degeneracy. The condition (\ref{3.8})
of the  finite degree degeneracy  is essential in our arguments. It
is not clear whether results in this section still hold without this
assumption.

An example of $\Omega$ and $K$ is given by
$$\Omega=\{(x,y); \ |x|<y<1\},$$
and $$K(x,y)=y^2-x^2.$$ Obviously, (\ref{3.6z}) and (\ref{3.8})
are satisfied for $\kappa(x)=|x|$ and $d=2$.

Our intention is to derive energy estimates. We first derive an
estimate on $H^1$-norms.

\begin{lemma}\label{Lemma-L2} Let $a, b_1, b_2, c$ and $K$ be $C^d$-functions
in $\bar\Omega$ satisfying (\ref{3.5})-(\ref{3.8}) and $u$ be an
$H^{d+3}$-solution of (\ref{3.4})-(\ref{3.4z}) for $\varphi\in
H^{d+2}(\partial_b\Omega)$, $\psi\in H^{d+1}(\partial_b\Omega)$
and $f\in H^{d+1}(\Omega)$. Then
\begin{equation}\label{h1}\|u\|_{1,\Omega}
\le
C\big(\|\varphi\|_{d+2,\partial_b\Omega}+\|\psi\|_{d+1,\partial_b\Omega}
+\|f\|_{d+1,\Omega}\big),\end{equation} where $C$ is a positive
constant depending on $\lambda$, $\Lambda$, $C_b$, $C_K$, $\eta_0$
and the $C^d$-norms of $a, b_1, b_2, c$ and $K$.
\end{lemma}

We note that (\ref{h1}) exhibits a loss of derivatives and such a
loss depends on the degree to which coefficients degenerate along
the boundary.

\begin{proof}
Multiplying $2e^{-\mu y}u_y/K$ to (\ref{3.4}), we get
\begin{align*}
&\partial_y\big(e^{-\mu y}(\frac{u_y^2}K +au_x^2)\big)
-2\partial_x(e^{-\mu y}au_xu_y)+e^{-\mu
y}\big(\mu+\frac{K_y}K\big)\frac{u_y^2}K
+e^{-\mu y}\big(\mu-\frac{a_y}{a}\big) au_x^2\\
=&-2e^{-\mu y}a_xu_xu_y-2e^{-\mu y}b_2\frac{ u_y^2}K-2e^{-\mu
y}b_1\frac{u_xu_y}K -2e^{-\mu y}c\frac{ uu_y}K+2e^{-\mu
y}\frac{u_yf}K.\end{align*} By combining with
\begin{equation*}\partial_y\big(e^{-\mu
y}\frac{ u^2}K\big)+e^{-\mu
y}\big(\mu+\frac{K_y}K\big)\frac{u^2}K=2e^{-\mu
y}\frac{uu_y}K,\end{equation*} we obtain
\begin{align}\label{y.1}\begin{split}
&\partial_y\big(e^{-\mu y}(\frac{u^2}K +\frac{u_y^2}K+au_x^2)\big)
-2\partial_x(e^{-\mu y}au_xu_y) \\
&\quad+e^{-\mu y}\big(\mu+\frac{K_y}K\big)(\frac{u^2}K
+\frac{u_y^2}K)
+e^{-\mu y}\big(\mu -\frac{a_y}a\big)au_x^2\\
=&-2e^{-\mu y}a_xu_xu_y-2e^{-\mu y}\frac {b_1}K
u_xu_y\\
&\quad -2e^{-\mu y}b_2\frac{u_y^2}{K}+2e^{-\mu
y}(1-c)\frac{uu_y}K+2e^{-\mu y}\frac{u_yf}K.\end{split}\end{align}
We point out again that $\partial_yK\ge 0$ by (\ref{3.6z}). We
first consider the second term in the right hand side. By
(\ref{3.6q}) and (\ref{3.6z}), we have
$$\frac {|b_1|}{\sqrt K}\le C_b\big(C_K\sqrt{\frac{K_y}K}+1\big).$$
By the Cauchy inequality, we get \begin{align*}\big|2e^{-\mu y}\frac
{b_1}K u_xu_y\big|&=\big|2e^{-\mu y}\frac {b_1}{\sqrt
K}\cdot\frac{u_y}{\sqrt K}u_x\big|\le e^{-\mu y}\big(\varepsilon
\frac{b_1^2}{K}\cdot\frac{u_y^2}{K}+\frac1\varepsilon u_x^2\big)\\
&\le e^{-\mu y}\big(2\varepsilon
C_b^2(C_K^2\frac{K_y}{K}+1)\cdot\frac{u_y^2}{K}+\frac1\varepsilon
u_x^2\big),\end{align*} for any $\varepsilon>0$.  By choosing
$\varepsilon>0$ small enough and applying the Cauchy inequality to
other terms in the right hand side of (\ref{y.1}), we obtain
\begin{align*}
&\partial_y\big(e^{-\mu y}(\frac{u^2}K +\frac{u_y^2}K+au_x^2)\big)
-2\partial_x(e^{-\mu y}au_xu_y) \\
&\quad+(\mu-\mu_0)e^{-\mu y}(\frac{u^2}K +\frac{u_y^2}K+au_x^2)\le
e^{-\mu y}\frac{f^2}K,\end{align*} where  $\mu_0$ is a positive
constant depending only on $\inf a$, $|a|_{C^1}$,
$|b_2|_{L^\infty}$, $|c|_{L^\infty}$, $C_b$ and $C_K$. By a simple
integration, we have
\begin{align*} (\mu-\mu_0)&\int_\Omega
e^{-\mu y}\big(\frac{u^2}K+\frac{u_y^2}K+au_x^2\big)  \le
\int_\Omega e^{-\mu y}\frac{f^2}K\\
&+ \int_{\partial_b\Omega}\frac{e^{-\mu y}}{\sqrt{1+\kappa_x^2}}\
(\frac{u^2}K+\frac{u_y^2}K+au_x^2 +2a\kappa_xu_xu_y),
\end{align*}
where the integral over $\partial_t\Omega$, having the correct sign,
is already dropped. By the Cauchy inequality and (\ref{3.7}), we get
$$2a|\kappa_xu_xu_y|\le
\frac{u_y^2}{K}+{aK\kappa_x^2}\cdot au_x^2\le
\frac{u_y^2}K+au_x^2.$$ Therefore,  by (\ref{3.5}) and taking
$\mu$ large enough, we obtain
\begin{equation}\label{y.6c}\int_\Omega \big(\frac{u^2}K
+\frac{u_t^2}K+ u_x^2\big)  \le C\int_{\partial_b\Omega}
\frac{1}{\sqrt{1+\kappa_x^2}} \big(\frac{u^2}K
+\frac{u_y^2}K+u_x^2\big)+ \int_\Omega \frac{f^2}{K}.
\end{equation}
We should note that the boundary integral in the right hand side of
(\ref{y.6c}) makes sense only when $u=u_y=0$ on $\partial_b\Omega$
if $K=0$ on $\partial_b\Omega$.

To eliminate $1/K$ from the right-hand side of (\ref{y.6c}), we
introduce an auxiliary function. It is easy to see that there
exists a $v\in H^{d+2}(\Omega)$ such that
$$D^\alpha v=D^\alpha u\quad\text{on }\partial_b\Omega\ \text{ for any
}|\alpha|\le d+1,$$ and
\begin{equation}\label{y.11}\|v\|_{d+2,\Omega}\le
C\sum_{|\alpha|\le d+2}\|D^\alpha
u\|_{0,\partial_b\Omega}.\end{equation} Obviously, $v$ satisfies
$$v=\varphi, \ v_y=\psi \quad\text{on
}\partial_b\Omega,$$ and
\begin{equation}\label{y.14} \partial_y^i(f-Lv)=0 \quad\text{on
}\partial_b\Omega,\text{ for any }i=0,1,\cdots,d-1.\end{equation}
Then we have
\begin{align*} &L(u-v)=f-Lv\quad\text{in }\Omega,\\
&u-v=0,\ (u-v)_y=0\quad\text{on }\partial_b\Omega.\end{align*} By
applying (\ref{y.6c}) to $u-v$, we obtain
\begin{equation}\label{y.12}\int_\Omega
\big(\frac{(u-v)^2}K+\frac{(u_y-v_y)^2}{K}+ (u_x-v_x)^2\big) \le
C\int_\Omega \frac{(f-Lv)^2}{K}.\end{equation} With (\ref{4.7}), we
have
\begin{equation}\label{y.13}\int_\Omega \big(u^2+u_y^2+ u_x^2\big)
\le C\int_\Omega \big(v^2+v_y^2+ v_x^2\big) +C\int_\Omega
\frac{(f-Lv)^2}{K}.\end{equation} Next, we eliminate the factor
$1/K$ in the last integral in (\ref{y.13}). With (\ref{y.14}) and
(\ref{3.8}), a simple calculation yields
$$\big((f-Lv)(x,y)\big)^2\le
\big(y-\kappa(x)\big)^d\int_{\kappa(x)}^y\big(\partial_y^d(f-Lv)(x,t)\big)^2dt
\quad\text{for
any }(x,y)\in \Omega,$$ and
$$\int_\Omega \frac{(f-Lv)^2}{K}\le C\int_\Omega
\big(\partial_y^d(f-Lv)(x,s)\big)^2\le
C\big(\|f\|_{d,\Omega}+\|v\|_{d+2,\Omega}\big)^2.$$ Hence, we obtain
\begin{equation}\label{y.15}\|u\|_{1,\Omega} \le C\|v\|_{d+2,\Omega}+C\|f\|_{d,\Omega}.
\end{equation}
 With the help of
(\ref{y.11}), (\ref{3.4}) and the trace theorem, we get
\begin{align*}&\|v\|_{d+2,\Omega}\le C\sum_{|\alpha|\le d+2}\|D^\alpha u\|_{0,\partial_b\Omega}\\
\le&
C\big(\|\varphi\|_{d+2,\partial_b\Omega}+\|\psi\|_{d+1,\partial_b\Omega}
+\|f\|_{d,\partial_b\Omega}\big)\\
\le&
C\big(\|\varphi\|_{d+2,\partial_b\Omega}+\|\psi\|_{d+1,\partial_b\Omega}
+\|f\|_{d+1,\Omega}\big),\end{align*} where $C$ is a positive
constant depending only on the $C^d$-norms of $a, b_1, b_2, c$ and
$K$. This implies (\ref{h1}) easily.
\end{proof}

\begin{rmrk}\label{Remark-Lemma-L2} It is clear that we have
\begin{equation}\label{y.13a}\int_\Omega
\big(\frac{(u-v)^2}K+\frac{(u_y-v_y)^2}{K}+ (u_x-v_x)^2\big) \le
C\big(\|\varphi\|_{d+2,\partial_b\Omega}+\|\psi\|_{d+1,\partial_b\Omega}
+\|f\|_{d+1,\Omega}\big)^2.\end{equation} This will be used in the
proof of Lemma \ref{Lemma-Hs} below.
\end{rmrk}

Next, we derive estimates of derivatives of $u$.

\begin{lemma}\label{Lemma-Hs} For an integer $m\ge 1$,
let $a, b_1, b_2, c$ and $K$ be $C^{m+d-1}$-functions in
$\bar\Omega$ satisfying (\ref{3.5})-(\ref{3.8}) and $u$ be an
$H^{m+d+2}$-solution of (\ref{3.4})-(\ref{3.4z}) for $\varphi\in
H^{m+d+1}(\partial_b\Omega)$, $\psi\in H^{m+d}(\partial_b\Omega)$
and $f\in H^{m+d}(\Omega)$. Then
\begin{equation}\label{h2}\|u\|_{m,\Omega}
\le
C\big(\|\varphi\|_{m+d+1,\partial_b\Omega}+\|\psi\|_{m+d,\partial_b\Omega}
+\|f\|_{m+d,\Omega}\big),\end{equation} where $C$ is a positive
constant depending on $\lambda$, $\Lambda$, $C_b$, $C_K$, $\eta_0$
and the $C^{m+d-1}$-norms of $a, b_1, b_2, c$ and $K$.
\end{lemma}

\begin{proof} We prove by induction. We note that Lemma \ref{Lemma-L2}
corresponds the case $m=1$. Let $s$ be a positive integer $\le
m-1$. Apply $\partial_x^s$ to (\ref{3.4}) to get
\begin{equation}\label{q1} L_s(\partial_x^s u)=f_s,\end{equation}
where
$$L_s=\partial_{yy}-aK\partial_{xx}+\big(b_1-s(aK)_x\big)\partial_x+b_2\partial_y
+\big(c+sb_{1,x}-\frac{s(s-1)}2(aK)_{xx}\big),$$ and
\begin{align*}
f_s=&\partial_x^sf+\sum_{i=2}^{s-1}c_{s,i}'\partial_x^{s+2-i}(aK)\partial_x^iu
+\sum_{i=1}^{s-1}c_{s,i}''\partial_x^{s+1-i}b_1\partial_x^iu\\
&\quad +\sum_{i=0}^{s-1}c_{s,i}'''\partial_x^{s-i}c\partial_x^iu
+\sum_{i=0}^{s-1}c_{s,i}''''\partial_x^{s+2-i}b_2\partial_x^i\partial_yu,\end{align*}
for some constants $c_{s,i}', c_{s,i}'', c_{s,i}'''$ and
$c_{s,i}''''$. We will write
$$f_s=\partial_x^sf+\sum_{i=0}^{s-1}\Gamma_{si}'\partial_x^iu
+\sum_{i=0}^{s-1}\Gamma_{si}''\partial_x^i\partial_yu.$$ We should
note that $L_s$ has the same structure as $L$. As in the proof of
Lemma \ref{Lemma-L2}, we construct a function $v_s\in
H^{d+2}(\Omega)$ such that
$$D^\alpha v_s=D^\alpha (\partial_x^su)\quad\text{on }\partial_b\Omega \text{ for any
}|\alpha|\le d+1,$$ and
$$\|v_s\|_{d+2,\Omega}\le C\sum_{|\alpha|\le d+2}\|D^\alpha
(\partial_x^su)\|_{0,\partial_b\Omega}.$$ Similar to (\ref{y.12}),
we have
\begin{equation}\label{q.2}\int_\Omega \big(\frac{(\partial_x^su-v_s)^2}K
+\frac{(\partial_x^s\partial_yu-\partial_yv_s)^2}{K}+
(\partial_x^{s+1}u-\partial_xv_s)^2\big) \le C\int_\Omega
\frac{(f_s-L_sv_s)^2}{K},\end{equation} where $C$ is positive
constant depending only on $\inf a$, $|a|_{C^2}$, $|K|_{C^2}$,
$|b_1|_{C^1}$, $|b_2|_{C^1}$, $|c|_{L^\infty}$, $C_b$ and $C_K$.
We write
$$f_s-L_sv_s=\tilde f_s+\sum_{i=0}^{s-1}\Gamma_{si}'(\partial_x^iu-v_i)
+\sum_{i=0}^{s-1}\Gamma_{si}''\partial_y(\partial_x^iu-v_i),$$ where
$v_0,\cdots, v_{s-1}$ are constructed for $u, \cdots,
\partial_x^{s-1}u$ as $v_s$ for $\partial_x^su$, and
$$\tilde f_s=(\partial_x^sf-L_sv_s)+\sum_{i=0}^{s-1}\Gamma_{si}'v_i
+\sum_{i=0}^{s-1}\Gamma_{si}''\partial_yv_i.$$ This implies
$$\int_\Omega \frac{(f_s-L_sv_s)^2}{K}\le
C\big(\int_\Omega\frac{\tilde
f_s^2}{K}+\sum_{i=0}^{s-1}\int_\Omega\frac{(\partial_x^iu-v_i)^2}{K}
+\sum_{i=0}^{s-1}\int_\Omega\frac{(\partial_x^i\partial_yu-\partial_yv_i)^2}{K}\big),$$
where $C$ is a positive constant depending only on the $C^s$-norms
of $aK, b_1, b_2$ and $c$. Note
$$\partial_y^i\tilde f_s=0\quad\text{on }\partial_b\Omega\ \text{for
}i=0,\cdots, d-1.$$ Therefore, we get
\begin{align*}&\int_\Omega\frac{\tilde f_s^2}{K}
\le\int_\Omega\big(\partial_y^d\tilde
f\big)^2\\
\le&\int_\Omega|\partial_y^d(\partial_x^sf-L_sv_s)|^2+\sum_{i=0}^{s-1}
\int_\Omega|\partial_y^d(\Gamma_{si}'v_i)|^2
+\sum_{i=0}^{s-1}\int_\Omega|\partial_y^d(\Gamma_{si}''\partial_yv_i)|^2\\
\le& C\big(\|f\|_{s+d,\Omega}^2+\|v_s\|_{d+2,\Omega}^2
+\sum_{i=0}^{s-1}\|v_i\|_{d+1,\Omega}^2\big),\end{align*} where $C$
is a positive constant depending only on the $C^{s+d}$-norms of $aK,
b_1, b_2$ and $c$. For each $i=0,\cdots, s$, we have
\begin{align*}&\|v_i\|_{d+2,\Omega}\le C\sum_{|\alpha|\le d+2}\|D^\alpha
\partial_x^iu\|_{0,\partial_b\Omega}\\
\le
&C\big(\|\varphi\|_{i+d+2,\partial_b\Omega}+\|\psi\|_{i+d+1,\partial_b\Omega}
+\|f\|_{i+d,\partial_b\Omega}\big)\\
\le&
C\big(\|\varphi\|_{i+d+2,\partial_b\Omega}+\|\psi\|_{i+d+1,\partial_b\Omega}
+\|f\|_{i+d+1,\Omega}\big),\end{align*} where $C$ depends on the
$C^{i+d}$-norms of $a, b_1, b_2, c$ and $K$. In summary, we obtain
\begin{align}\label{q3}\begin{split}&\int_\Omega \big(\frac{(\partial_x^su-v_s)^2}K
+\frac{(\partial_y\partial_x^su-\partial_yv_s)^2}{K}+
(\partial_x^{s+1}u-\partial_xv_s)^2\big)\\ \le&
C\big(\sum_{i=0}^{s-1}\int_\Omega\big(\frac{(\partial_x^iu-v_i)^2}{K}
+\frac{(\partial_y\partial_x^iu-\partial_yv_i)^2}{K}\big)\\
&\quad
+C\big(\|\varphi\|^2_{s+d+2,\partial_b\Omega}+\|\psi\|^2_{s+d+1,\partial_b\Omega}
+\|f\|^2_{s+d+1,\Omega}\big),
\end{split}\end{align} where $C$ depends on the $C^{s+d}$-norms of $a,
b_1, b_2, c$ and $K$. By a simple induction starting from
(\ref{y.13a}), we obtain
\begin{align*}&\int_\Omega \big((\partial_x^su)^2
+(\partial_y\partial_x^su)^2+ (\partial_x^{s+1}u)^2\big)\\
\le&
C\big(\|\varphi\|^2_{s+d+2,\partial_b\Omega}+\|\psi\|^2_{s+d+1,\partial_b\Omega}
+\|f\|^2_{s+d+1,\Omega}\big).\end{align*} All other derivatives of
$u$ of order $s+1$ can be obtained from (\ref{3.4}).
\end{proof}

Now we prove the main result in this section.

\begin{thrm}\label{Thm-existenceDegen} For an integer $m\ge 2$,
let $a, b_1, b_2, c$ and $K$ be $C^{m+d-1}$-functions in
$\bar\Omega$ satisfying (\ref{3.5})-(\ref{3.8}) and $\varphi\in
H^{m+d+1}(\partial_b\Omega)$, $\psi\in H^{m+d}(\partial_b\Omega)$
and $f\in H^{m+d}(\Omega)$. If $\mathcal C_i(\varphi,\psi,f)$ holds
for $i=1,\cdots,m+d-2$, then (\ref{3.4})-(\ref{3.4z}) admits a
unique $H^{m+d+2}(\Omega)$-solution $u$ and such a $u$ satisfies
\begin{equation}\label{h3}\|u\|_{m,\Omega}
\le
C\big(\|\varphi\|_{m+d+1,\partial_b\Omega}+\|\psi\|_{m+d,\partial_b\Omega}
+\|f\|_{m+d,\Omega}\big),\end{equation} where $C$ is a positive
constant depending on $\lambda$, $\Lambda$, $C_b$, $C_K$, $\eta_0$
and the $C^{m+d-1}$-norms of $a, b_1, b_2, c$ and $K$. Moreover, if
$a, b_1, b_2, c, K$ and $f$ are $H^s(\bar \Omega)$ and $\varphi$ and
$\psi$ are $H^s(\partial_b\Omega)$ for any $s\ge 1$ and $\mathcal
C_i(\varphi,\psi,f)$ holds for any $i\ge 1$, then $u$ is smooth and
satisfies (\ref{h3}) for any $m\ge 1$. \end{thrm}

\begin{proof} For a positive sequence $\varepsilon\to0$, we
consider  an  equation  of the following form
\begin{equation}\label{h4}L_\varepsilon u\equiv u_{yy}-a(K+\varepsilon)u_{xx}
+b_1u_x+b_2u_y+cu=f_\varepsilon\quad\text{in }\Omega
\end{equation} with the Cauchy data
\begin{equation}\label{h5} u=\varphi_\varepsilon, \ u_y=\psi_\varepsilon
\quad\text{on }\partial_b\Omega,\end{equation} where
$\varphi_\varepsilon$, $\psi_\varepsilon$ and $f_\varepsilon$ are
chosen so that
$$\varphi_\varepsilon\to\varphi\text{ in }H^{m+d+1}(\partial_b\Omega),\quad
\psi_\varepsilon\to\psi\text{ in }H^{m+d}(\partial_b\Omega),\quad
f_\varepsilon\to f\text{ in }H^{m+d}(\Omega),$$ and
$$C_i(\varphi_\varepsilon,\psi_\varepsilon,f_\varepsilon)\text{
holds for $L_\varepsilon$ for any $i=1,\cdots, m+d-2$}.$$ We note
that $L_\varepsilon$ in (\ref{h4}) is strictly hyperbolic in
$\bar\Omega$. By Theorem \ref{Thrm-existence}, (\ref{h4})-(\ref{h5})
admits a solution $u_\varepsilon\in H^{m+d}(\Omega)$. By Lemma
\ref{Lemma-Hs},  $u_\varepsilon$ satisfies
\begin{equation*}\|u_\varepsilon\|_{m,\Omega}
\le C\big(\|\varphi_\varepsilon\|_{m+d+1,\partial_b\Omega}
+\|\psi_\varepsilon\|_{m+d,\partial_b\Omega}
+\|f_\varepsilon\|_{m+d,\Omega}\big),\end{equation*} where $C$ is
a positive constant depending on $\lambda$, $\Lambda$, $C_b$,
$C_K$, $\eta_0$ and the $C^{m+d-1}$-norms of $a, b_1, b_2, c$ and
$K$. We finish the proof by letting $\varepsilon\to0$.
\end{proof}

\section{Proof of Theorem \ref{Theorem0.1}}\label{Section-Proof}

In this section, we will prove a result of which Theorem
\ref{Theorem0.1} is a special case.

We consider an equation  of the following form
\begin{equation}\label{r1}Lu\equiv u_{yy}+aKu_{xx}
+b_1u_x+b_2u_y+cu=f\quad\text{in }B_2\subset\mathbb R^2,
\end{equation}
where $a, b_1, b_2, c$ and $K$ are smooth in $B_2$. We always
assume
\begin{equation}\label{r2}a\ge\lambda  \quad \text{in
}B_2,\end{equation} for a positive constant $\lambda$. Concerning
$K$, we assume
\begin{align}\label{r3}\begin{split} &\text{$\{K=0\}$ consists of two
curves given by smooth functions $y=\gamma_i(x)$, }\\
&\quad\text{where $y=\gamma_1(x)$ is decreasing and
$y=\gamma_2(x)$ is increasing and}\\
&\quad\text{$\gamma_1(0)=0$, $\gamma_2(0)=0$,
$\gamma_1'(0)\neq\gamma_2'(0)$}.\end{split}\end{align} By setting
$$\kappa_1(x)=\max\{\gamma_1(x), \gamma_2(x)\},\quad \kappa_2(x)=
\min\{\gamma_1(x), \gamma_2(x)\},$$ we note that $\kappa_1(x)$ and
$\kappa_2(x)$ are smooth at any $x\neq 0$, $\kappa_i(0)=0$ and
$\kappa_1(x)>0$ and $\kappa_2(x)<0$ for any $x\neq 0$. Obviously,
$y=\kappa_1(x)$ and $y=\kappa_2(x)$ divide $B_2$ into four
regions. We denote by $\Omega_+$ and $\Omega_-$ the
union of the two regions containing the $x$-coordinate axis and the
$y$-coordinate axis, respectively. We further assume that
\begin{equation}\label{r4}
\text{$K>0$ in $\Omega_+$ and $K<0$ in $\Omega_-$}.\end{equation}
Moreover, we assume that
\begin{equation}\label{r5}K_x^2\le C_K^2|K_y|\quad\text{in }\Omega_-,
\end{equation}
and
\begin{equation}\label{r6} \big|y-\kappa(x)\big|^d\le
C_K|K(x,y)|\quad\text{for any }(x,y)\in\Omega_-,\end{equation} where
$C_K$ is a positive constant and $d$ is a positive integer.
Concerning coefficients $b_1$ and $c$, we assume
\begin{equation}\label{r7}|b_1|\le C_b\big(\sqrt{K}+|K_x|\big) \quad\text{in }\Omega,
\end{equation} for a positive constant  $C_b$ and
\begin{equation}\label{r8} c\le 0\quad\text{in
}\Omega_+.\end{equation} We note that (\ref{r5}) and (\ref{r6})
are assumed only in $\Omega_-$ and (\ref{r8}) only in $\Omega_+$.

Now we explain briefly the roles of these assumptions. The curves
$y=\gamma_1(x)$ and $y=\gamma_2(x)$ divide $B_2$ into four regions, in
two of which (\ref{r1}) is elliptic and in another two (\ref{r1}) is
hyperbolic by (\ref{r4}). For any one of the regions, the origin is
an angular point. For any hyperbolic region, the part of the
boundary containing the origin is space-like. The assumption
(\ref{r7}) is the so-called {\it Levy condition}. It is needed in
both elliptic regions and hyperbolic regions. The condition
(\ref{r8}) is used to ensure the existence of solutions of the
Dirichlet problem in elliptic regions. The assumptions (\ref{r5})
and (\ref{r6}) are needed to overcome the degeneracy in the
hyperbolic regions.

For equation (\ref{0.1}) in Theorem \ref{Theorem0.1}, we have
$K(x,y)=x^2-y^2$, $\kappa_1(x)=|x|$, $\kappa_2(x)=-|x|$ and $d=2$.

We now present a result more general than Theorem \ref{Theorem0.1}
and only formulate it for the infinite differentiability.

\begin{thrm}\label{TheoremMain} Let $a, b_1, b_2, c$ and $K$ be
smooth functions in $B_2\subset \mathbb R^2$ satisfying
(\ref{r2})-(\ref{r8}). Then for any smooth function $f$ in $B_2$,
there exists a smooth solution $u$ of (\ref{r1}) in $B_1$. Moreover,
for any nonnegative integer $s$, $u$ satisfies
\begin{equation}\label{main}\|u\|_{H^s(B_1)}\le
c_s\|f\|_{H^{s+d+3}(B_2)},\end{equation} where $c_s$ is a positive
constant depending only on $s$, $\lambda$, $C_K$, $C_b$, the
$C^1$-norm of $\gamma_i$, $i=1,2$,  and the $C^{s+d+2}$-norms of
$a, b_1, b_2, c$ and $K$.\end{thrm}

\begin{proof} Throughout the proof, we denote by $C_s$ a positive constant
depending only on $s$, $\lambda$, $C_K$, $C_b$, the $C^1$-norm of
$\gamma_i$, $i=1,2$,  and the $C^{s}$-norms of $a, b_1, b_2, c$
and $K$.

We first smoothen the corner of $\partial\Omega_+$ at $\partial B_2$
and consider (\ref{r1}) in $\Omega_+$. By Theorem \ref{Theorem4.1},
there exists a smooth solution $u$ of (\ref{r1}) in $\Omega_+$ with
$u=0$ on $\partial\Omega_+$. Moreover, for any integer $s\ge 1$, $u$
satisfies
\begin{equation}\label{r11}\|u\|_{H^s(\Omega_+)}\le
C_s\|f\|_{H^{s+1}(\Omega_+)}.\end{equation} By the trace theorem,
we obtain
\begin{equation}\label{r12}\sum_{|\alpha|\le s}
\|D^\alpha u\|_{L^2(\partial\Omega_+)}\le
C_{s+1}\|f\|_{H^{s+2}(\Omega_+)}.\end{equation}

Next, we assume $y=1$ intersects $y=\kappa_1(x)$ for a positive
$x$ and a negative $x$ in $B_2$. If not, we may extend $K$
appropriately outside $B_2$ to achieve this. Now we set
$$\Omega_{-1}=\Omega_-\cap \{0<y<1\},$$
and
$$\varphi=0,\quad \psi=u_y\quad\text{on }\partial_b \Omega_{-1},$$
where $\partial_b\Omega_{-1}$ is the lower portion of
$\partial\Omega_{-1}$. We consider (\ref{r1}) in $\Omega_{-1}$
with the Cauchy data
\begin{equation}\label{r13} u=\varphi, \
u_y=\psi\quad\text{on }\partial_b\Omega_{-1}.\end{equation} Since
$\varphi$ and $\psi$ are boundary values of a smooth solution $u$
in $\bar{\Omega}_+$, it is easy to check that compatibility
conditions $\mathcal C_i(\varphi, \psi, f)$ are satisfied for any
$i\ge 1$ by Lemma \ref{Lemma-extension0}. By Theorem
\ref{Thm-existenceDegen}, there exists a smooth solution $u$ of
(\ref{r1}) in $\Omega_{-1}$ satisfying (\ref{r13}). Moreover, for
any integer $s\ge 1$, $u$ satisfies
\begin{equation*}\|u\|_{H^s(\Omega_{-1})}
\le C_{s+d}\big(\|\varphi\|_{H^{s+d+1}(\partial_b\Omega_{-1})}
+\|\psi\|_{H^{s+d}(\partial_b\Omega_{-1})}
+\|f\|_{H^{s+d}(\Omega_{-1})}\big).\end{equation*} With
(\ref{r12}), we have easily
\begin{equation*}\|u\|_{H^s(\Omega_{-1})}
\le
C_{m+d+2}\big(\|f\|_{H^{s+d+3}(\Omega_+)}+\|f\|_{H^{s+d}(\Omega_{-1})}\big).
\end{equation*}
A similar argument can be applied to
$$\Omega_{-2}=\Omega_-\cap \{-1<y<0\}.$$
Therefore we obtain a function $u$ which is a smooth solution of
(\ref{r1}) in $\Omega_+\cap B_1$ and $\Omega_-\cap B_1$. It is
easy to see that $u$ is smooth across $\partial\Omega_+\cap B_1$
and  especially at the origin. The estimate (\ref{main}) also
follows easily.
\end{proof}

\begin{rmrk}\label{Remark-forTheorem5.1}
We also note that $c\le 0$ in (\ref{r8}) can
be replaced by $c\le\varepsilon$ for $\varepsilon>0$ sufficiently small.
Refer to Remark \ref{Remark2-forTheorem4.1}.
\end{rmrk}

The estimate (\ref{main})
is not sufficient for the iteration process when solving the nonlinear equations.
For this, we need a stronger estimate.

\begin{thrm}\label{TheoremMain-Linear} Let $a, b_1, b_2, c$ and $K$ be
smooth functions in $B_2\subset \mathbb R^2$ satisfying
(\ref{r2})-(\ref{r8}). Then for any smooth function $f$ in $B_2$,
there exists a smooth solution $u$ of (\ref{r1}) in $B_1$. Moreover,
for any nonnegative integer $s$, $u$ satisfies
\begin{equation}\label{main-linear}\|u\|_{H^s(B_1)}
\le c_s\left(||f||_{H^{s+d+3}(B_1)}
+\Lambda_s||f||_{H^{d+3}(B_1)}\right),
\end{equation}
where $c_s$ is a constant depending only on $s$,
$\lambda$, $C_K$, $C_b$, the
$C^1$-norm of $\gamma_i$, $i=1,2$, and where

$$\Lambda_s=\|a\|_{H^{s+d+4}(B_1)}+\sum_{i=1}^2\|b_i\|_{H^{s+d+4}(B_1)}
+\|c\|_{H^{s+d+4}(B_1)}+\|K\|_{H^{s+d+4}(B_1)}+1.$$\end{thrm}

We note that all estimates in Sections
\ref{Section-Elliptic}-\ref{Section-HyperbolicEstimates}
are standard energy estimates.  Hence, we obtain (\ref{main-linear})
with the help of interpolation inequalities. We skip the details.

\section{Proof of Theorem
\ref{Theorem-Nonlinear}}\label{Sec-Iterations}

In this section, we will prove a result of which Theorem
\ref{Theorem-Nonlinear} is a special case.

Consider an equation  of the following form
\begin{equation}\label{eq-nr1}
\det(D^2u)=K(x,y)\psi(x,y,u,Du)\quad\text{in }B_1\subset\mathbb R^2,
\end{equation}
where $K$ is smooth in $B_1$ and $\psi$ is smooth
in $B_1\times\mathbb R\times\mathbb R^2$. We always
assume
\begin{equation}\label{eq-nr2}\psi\ge\lambda  \quad \text{in
}B_1\times\mathbb R\times\mathbb R^2,\end{equation}
for a positive constant $\lambda$. Concerning
$K$, we assume that $K$ satisfies (\ref{r3})-(\ref{r6}). In other words, we
assume
\begin{align}\label{eq-nr3}\begin{split} &\text{$\{K=0\}$ consists of two
curves given by smooth functions $y=\gamma_i(x)$, }\\
&\quad\text{where $y=\gamma_1(x)$ is decreasing and
$y=\gamma_2(x)$ is increasing and}\\
&\quad\text{$\gamma_1(0)=0$, $\gamma_2(0)=0$,
$\gamma_1'(0)\neq\gamma_2'(0)$}.\end{split}\end{align} By setting
$$\kappa_1(x)=\max\{\gamma_1(x), \gamma_2(x)\},\quad \kappa_2(x)=
\min\{\gamma_1(x), \gamma_2(x)\},$$ we note that $\kappa_1(x)$ and
$\kappa_2(x)$ are smooth at any $x\neq 0$, $\kappa_i(0)=0$ and
$\kappa_1(x)>0$ and $\kappa_2(x)<0$ for any $x\neq 0$. Obviously,
$y=\kappa_1(x)$ and $y=\kappa_2(x)$ divide $B_1$ into four
regions. We denote by $\Omega_+$ and $\Omega_-$ the
union of the two regions containing the $x$-coordinate axis and the
$y$-coordinate axis, respectively. We further assume
\begin{equation}\label{eq-nr4}
\text{$K>0$ in $\Omega_+$ and $K<0$ in $\Omega_-$}.\end{equation}
Moreover, we assume
\begin{equation}\label{eq-nr5}K_x^2\le C_K^2|K_y|\quad\text{in }\Omega_-,
\end{equation}
and
\begin{equation}\label{eq-nr6} \big|y-\kappa(x)\big|^d\le
C_K|K(x,y)|\quad\text{for any }(x,y)\in\Omega_-,\end{equation} where
$C_K$ is a positive constant and $d$ is a positive integer.

We now point out the difference between the assumptions on $K$ for (\ref{r1})
and (\ref{eq-nr1}). For linear equations having the specific form of
(\ref{r1}), the conditions on $K$ are assumed with respect to this
{\it particular} coordinate system. However, the Monge-Amp\`{e}re operator is
invariant by orthogonal transformations. Hence, conditions on $K$ for
(\ref{eq-nr1}) in this section are assumed in {\it some} coordinate system.

We now present a result more general than Theorem \ref{Theorem-Nonlinear}
and only formulate it for the case of infinite differentiability.

\begin{thrm}\label{TheoremMain-Nonlinear} Let $\psi$ be
a smooth function satisfying (\ref{eq-nr2}) and let
$K$ be a smooth function in $B_1$ satisfying
(\ref{eq-nr3})-(\ref{eq-nr6}).
Then there exists a smooth solution
$u$ of (\ref{eq-nr1}) in $B_r$ for some $r\in (0,1)$.\end{thrm}

The proof of Theorem \ref{TheoremMain-Nonlinear} is based on
Nash-Moser iterations.
An important step in such an iteration process consists of appropriate estimates
for solutions of the linearized equations.
In the case of the degenerate Monge-Amp\`{e}re
equation (\ref{eq-nr1}), the
linearized equations are hard to classify.
A crucial observation by Han \cite{Han2007} is that
the linearization of Monge-Amp\`{e}re
equations can be decomposed into two parts, one of which has
type determined solely by $K$ and another
which may be considered as quadratic error with respect to the iteration process.

In the following, we denote points in $\mathbb R^2$ by
$(x_1, x_2)$ instead of $(x,y)$ and write $x=(x_1,x_2)\in\mathbb R^2$.
Set
\begin{equation}\label{i2.4}
\tilde {\mathcal{F}}(u)= \det (D^2u)-K \psi(x, u, Du).
\end{equation}
To proceed, we temporarily replace $x \in \mathbb{R}^2 $ by
$\tilde{x}\in \mathbb{R}^2 $, replace $\psi$ by $\tilde \psi$, and write
$\widetilde{\partial_i}$ instead of $\partial_{\tilde{x}_i}$.
Then (\ref{i2.4}) has the form
\begin{equation*}
\tilde {\mathcal{F}}(u)=\det(\widetilde{D}^2u)-K \tilde \psi(\tilde
x, u, \widetilde{D}u).
\end{equation*} All functions are evaluated at
$\widetilde{x}$.
For $\varepsilon>0$ set
$$\tilde{x}=\varepsilon^2 x, $$
and
\begin{equation*}u(\tilde{x})=\frac12\tilde x_1^2
+\varepsilon^{5}w\left(\frac{\tilde{x}}
{\varepsilon^2} \right).\end{equation*}
Now we evaluate $\tilde {\mathcal{F}}(u)$ in terms of $w$. Set
\begin{equation*}
\mathcal{F}(w;
\varepsilon)={\mathcal{F}}(w)=\frac{1}{\varepsilon}\tilde
{\mathcal{F}}(u),\end{equation*} or
\begin{equation}\label{i2.9}
\mathcal{F}(w)=\frac{1}{\varepsilon} \big\{
\det\big((1-\delta_{i2})\delta_{1j}+ \varepsilon
\partial_{ij} w\big)  - K \psi \big\},
\end{equation}
where
\begin{equation}\label{i2.10}
\psi(\varepsilon, x, w, Dw)=\tilde \psi\big(\varepsilon^2 x,
\frac12\varepsilon^4
x_1^2+\varepsilon^{5}w(x),
\varepsilon^2(1-\delta_{i2})x_i+\varepsilon^3\partial_iw(x)\big).
\end{equation}
Note that the arguments of $\tilde \psi$ are $\tilde x$, $u$ and
$\tilde Du$ in terms of $w$ in the $x$-coordinates. All known
functions are evaluated at $\tilde x=\varepsilon^2 x$. By taking
$\varepsilon$ small enough, we may assume $\mathcal{F}(w)$ is well
defined in
$B_1\subset \mathbb R^2$.
Letting $w=0$ in (\ref{i2.9}), we have
\begin{equation*}
\mathcal{F}(0)=-\frac{1}{\varepsilon} K \psi.
\end{equation*}
By $K=K(\varepsilon^2 x)$ and $K(0)=0$, there holds
\begin{equation*}\label{i2.11}
\mathcal{F}(0)=\varepsilon F_0(\varepsilon, x),
\end{equation*}
for some smooth function $F_0$ in $\varepsilon$ and $x$. We
also have
$$\psi(\varepsilon, x, w, Dw)\ge \lambda,$$
for any
$x\in B_1$, any $\varepsilon$ small and
any $w\in C^\infty(B_1)$.

Now we discuss the linearized operator $\mathcal{F}'(w)$ of
$\mathcal{F}$ at $w$. For convenience, we set
\begin{equation*}\label{i2.12}(\Phi_{ij})=
\big((1-\delta_{i2})\delta_{1j} +
\varepsilon
\partial_{ij} w\big).\end{equation*}
A straightforward
calculation yields
\begin{equation}\label{i2.13}
\mathcal{F}'(w)\rho=\Phi^{ij}\partial_{ij}\rho+a_i\partial_i\rho+a\rho,
\end{equation}
where $(\Phi^{ij})$ is the matrix of cofactors of $(\Phi_{ij})$, i.e.,
\begin{equation}\label{i2.13a}
\Phi^{11}=\varepsilon\partial_{22}w, \quad
\Phi^{12}=-\varepsilon\partial_{12}w, \quad
\Phi^{22}=1+\varepsilon\partial_{11}w,
\end{equation}
and
\begin{equation}\label{i2.14}
a_i=a_i(\varepsilon, x, w,
Dw)=-\varepsilon^2K\partial_{\tilde\partial_iu}\tilde \psi, \quad
a=a(\varepsilon, x, w, Dw)=-\varepsilon^4K\partial_u\tilde
\psi.\end{equation} As in (\ref{i2.10}),
$\partial_{\tilde\partial_iu}\tilde \psi$ and $\partial_u\tilde \psi$
are evaluated at $$\big(\varepsilon^2 x,
\frac12\varepsilon^4
x_1^2+\varepsilon^{5}w(x),
\varepsilon^2(1-\delta_{i2})x_i+\varepsilon^3\partial_iw(x)\big).$$
Obviously, $a_i$ and $a$ are smooth in $\varepsilon$, $x$, $w$ and $Dw$.

By (\ref{i2.9}), we have
\begin{equation}\label{i2.17}\det(\Phi^{ij})=
\varepsilon\mathcal{F}(w)+K\psi.
\end{equation} It is not clear how $K$ determines the type
of the linear operator $\mathcal{F}'(w)$ in (\ref{i2.13}). Next, we
shall introduce a new coordinate system and rewrite (\ref{i2.13}).

\begin{lemma}\label{Lemma-i2.2}
For any $\varepsilon\in(0, \varepsilon_0]$ and any smooth function
$w$ with $|w|_{C^2}\le 1$, there exists a transformation $T: B_1\to
T(B_1)$, smooth in $\varepsilon, x$, $D^2w$
and $D^3w$, of the form
\begin{equation}\label{i2.18}
x\mapsto y=(y_1(x), y_2(x))\end{equation} such
that in the new coordinates $y$ the operator $\mathcal{F}'(w)$ is
given by
\begin{align}\label{i2.99}\begin{split}
{\mathcal
F}'(w)\rho=&a_{22}\partial_{y_2y_2}\rho
+\big(K\psi+\varepsilon\mathcal{F}(w)\big)a_{11}\partial_{y_1y_1}\rho\\
&+\big(b_{10}K+
b_{11}\partial_{y_1}K+\varepsilon\tilde
b_{10}\mathcal{F}(w)+\tilde
b_{11}\partial_{y_1}(\mathcal{F}(w)\big)
\partial_{y_1}\rho+b_2\partial_{y_2}\rho+cK\rho,\end{split}\end{align} where
$a_{11}, a_{22}$, $b_{10}$, $b_{11}$, $\tilde b_{10}$, $\tilde b_{11}$,
$b_2$ and $c$ are smooth functions in $\varepsilon$, $y$,
$w$, $Dw$, $D^2w$, $D^3w$ and $D^4w$, with
\begin{equation*}\label{i2.21}
a_{ii}=1+O(\varepsilon)\quad\text{for }
i=1,2.\end{equation*} Moreover,
for $i=1, 2$, $y_i=y_i(x)$ in (\ref{i2.18}) satisfies
\begin{equation*}\label{i2.22}|y_i-x_i|\le
c\varepsilon,\end{equation*} and for any $s\ge 0$
\begin{equation*}\label{i2.23}\|y_i\|_{H^s}\le
c(1+\|w\|_{H^{s+2}}),\end{equation*} for some positive constant $c$.
\end{lemma}

This is Lemma 2.2 in
\cite{Han2007} for $n=2$ (page 430). The proof for $n=2$ is easy. We outline
the proof
for completeness.

\begin{proof}  By
(\ref{i2.13a}), we have
\begin{equation}\label{i2.25}
\Phi^{ij}=\delta_{i2}\delta_{j2}
+O(\varepsilon)\quad\text{for any }1\le i,j\le 2.
\end{equation}
First, we set
\begin{equation}\label{i2.28}
y_2=x_2.\end{equation} Next, we
consider the following equation for $y_1$
\begin{equation}\label{i2.26}\begin{split}
\Phi^{12}\partial_1y_1+\Phi^{22}\partial_2y_2&=0,\\
y_1(x_1, 0)&=x_1.\end{split}\end{equation} The
coefficient of $\partial_2y_1$ is given by $\Phi^{22},$ which is not
zero for small $\varepsilon$. Hence for small $\varepsilon$,
(\ref{i2.26}) always has a unique solution $y_1$ in $B_1$,
smooth in $\varepsilon$, $x$ and $D^2w$.
Moreover,
\begin{equation}\label{i2.27}
y_1(x)=x_1+O(\varepsilon).\end{equation} Obviously, $y=y(x)$ forms
a new coordinate system. This defines the transformation $T$ in
(\ref{i2.18}).

In the new coordinates $y$, the operator $\mathcal{F}'(w)$
has the following form
\begin{equation}\label{i2.29}
\mathcal{F}'(w)\rho=b_{ij}\partial_{y_iy_j}\rho+b_i\partial_{y_i}\rho
+a\rho,\end{equation} where
\begin{equation*}
b_{ij}=\sum_{k,l=1}^2\Phi^{kl}\partial_ky_i\partial_ly_j,
\end{equation*}
and
\begin{equation*}
b_i=\sum_{k,l=1}^2\Phi^{kl}\partial_{kl}y_i+
\sum_{k=1}^2a_k\partial_ky_i.\end{equation*}
We now claim that
\begin{equation}\label{i2.30}
b_{11}=\frac1{\Phi^{22}}\det(\Phi^{ij})(\partial_1y_1)^2,\quad
b_{12}=0,\quad
b_{22}=\Phi^{22},\end{equation}
and
\begin{equation}\label{i2.31}
b_1=\partial_1\big(\frac{\det(\Phi^{ij})}{\Phi^{22}}\partial_1y_1\big)
+\sum_{k=1}^2a_k\partial_ky_1,\quad
b_2=a_2.\end{equation}
To prove the claim, we note that the expressions for $b_{22}$ and $b_2$
follow from (\ref{i2.28}) and those for $b_{12}$ and $b_{11}$
follow from (\ref{i2.26}).
To calculate $b_1$, we have by (\ref{i2.25})
\begin{equation*}\label{2.46}\sum_{k=1}^2\partial_k\Phi^{kl}=0.\end{equation*}
Then the first term in $b_1$ in (\ref{i2.31}) can be written as
$$\sum_{k, l=1}^2\Phi^{kl}\partial_{kl}y_1=
\sum_{k, l=1}^{2}\partial_k(\Phi^{kl}\partial_{l}y_1).$$
Then the expression for $b_1$ follows again from (\ref{i2.26}).

By substituting (\ref{i2.30}) and (\ref{i2.31}) in (\ref{i2.29}), we have
\begin{align*}
{\mathcal
F}'(w)\rho&=\Phi^{22}\partial_{y_2y_2}\rho+
\frac1{\Phi^{22}}\det(\Phi^{ij})(\partial_1y_1)^2\partial_{y_1y_1}\rho\\ & +
\left(\partial_1\big(\frac{\det(\Phi^{ij})}{\Phi^{22}}\partial_1y_1\big)
+\sum_{k=1}^2a_k\partial_ky_1\right)
\partial_{y_1}\rho
+a_2\partial_{y_2}\rho
+a\rho.\end{align*}
Recalling (\ref{i2.14}), (\ref{i2.17}), (\ref{i2.25}) and (\ref{i2.27}), we
conclude the proof.
\end{proof}

Next, we write
$\mathcal{F}'(w)$ in (\ref{i2.99}) as
\begin{equation}\label{i2.19}
{\mathcal
F}'(w)\rho=\mathcal{L}(w)\rho+\varepsilon\mathcal{F}(w)\sum_{i,j=1}^{2}
\tilde a_{ij}
\partial_{ij}\rho+
\varepsilon\sum_{i,j=1}^{2}\big(\tilde b_{j0}\mathcal{F}(w)+\tilde
b_{ij}\partial_i(\mathcal{F}(w))\big)
\partial_{j}\rho,\end{equation}
where $\tilde a_{ij}$ and $\tilde b_{ij}$ are functions smooth
in $\varepsilon$, $x$, $D^2w$, $D^3w$ and
$D^4w$, and $\mathcal{L}(w)$ has the following form in $T(B_1)$
\begin{equation}\label{i2.20}
\mathcal{L}(w)\rho=a_{22}\partial_{y_2y_2}\rho
+a_{11}K\partial_{y_1y_1}\rho +(b_{1}K+\tilde b_{1}\partial_{y_1}K)
\partial_{y_1}\rho+b_2\partial_{y_2}\rho+cK\rho,\end{equation}
where $a_{11}, a_{22}$, $b_{1}$, $\tilde b_1$, $b_2$ and $c$ are functions
smooth in $\varepsilon$, $y$, $w$, $Dw$, $D^2w$,
$D^3w$ and $D^4w$. We point out that, in the new coordinate system
$(y_1, y_2)$ in (\ref{i2.18}), the operator
$\mathcal{L}(w)$ in (\ref{i2.20}) has a special
structure. Both $a_{11}$ and $a_{22}$ are positive and there is a factor
of $K$ in the coefficient of $\partial_{y_1y_1}\rho$. Hence, the
operator $\mathcal{L}(w)$ is elliptic if $K>0$ and hyperbolic if
$K<0$. We emphasize that the type of $\mathcal{L}(w)$ is
determined {\it solely} by $K$ and is independent of $w$, the function at which
the linearized operator is evaluated. This is crucial for the
iterations. Next, we note that the correction terms that were added in
(\ref{i2.19}) are {\it quadratic} in $\mathcal{F}(w)$ and $\rho$,
and their derivatives. Hence they can be relegated to the quadratic
error in the iteration process, that is, they may be ignored when solving the linearized equation.

By writing
\begin{equation*}
a_{22}^{-1}\mathcal{L}(w)\rho=\partial_{y_2y_2}\rho
+\frac{a_{11}}{a_{22}}K\partial_{y_1y_1}\rho +\frac{1}{a_{22}}(b_{1}K+\tilde b_{1}\partial_{y_1}K)
\partial_{y_1}\rho+\frac{b_2}{a_{22}}\partial_{y_2}\rho
+\frac{c}{a_{22}}K\rho,\end{equation*}
we note that the coefficients in the right hand side satisfy
(\ref{r2})-(\ref{r7}). (The notation here is different from that used in the previous
section.) The coefficient of $\rho$ may not be nonpositive. However,
it is small as $K=K(\varepsilon^2 x)$ and $K(0)=0$.
By Theorem \ref{TheoremMain-Linear} and Remark \ref{Remark-forTheorem5.1},
for any smooth function $f$ in $B_1$ and any $\varepsilon>0$ sufficiently small,
there exists a smooth function $\rho$ in $B_1$ such that
$$\mathcal L(w)(\rho\circ T^{-1})=f\circ T^{-1}\quad\text{in }T(B_1),$$
where $T$ is the transformation given by (\ref{i2.18}).
(In the following, we abuse notation and
simply write $\mathcal {L}(w)\rho=f$ in $B_1$.)
Moreover, if $\|w\|_{H^4(B_1)}\le 1$, then for any nonnegative integer $s$
\begin{equation}\label{eq-linear00}\|\rho\|_{H^s(B_1)}
\le c_s\left(||f||_{H^{s+d+3}(B_1)}
+(\|w\|_{H^{s+d+8}(B_1)}+1)||f||_{H^{d+3}(B_1)}\right),
\end{equation}
where $c_s$ is a constant depending only on $s$,
$\lambda$, $C_K$, the
$C^1$-norm of $\gamma_i$, $i=1,2$ and the $H^{s+d+8}$-norm of $K$.
Here, we use the fact that $a_{11}, a_{22}$, $b_{1}$, $b_2$ and $c$ are functions
smooth in $\varepsilon$, $y$, $w$, $Dw$, $D^2w$,
$D^3w$ and $D^4w$.

The proof of Theorem \ref{TheoremMain-Nonlinear} is based on
Nash-Moser iterations. A general
result for the existence of local smooth solutions is formulated in
\cite{Han-Hong2006}. (Refer to Theorem 7.4.1 on page 130 \cite{Han-Hong2006}.)
However, the linearized equations of (\ref{eq-nr1}) do not satisfy
the condition listed there. Specifically, solutions of the linearized
equations of (\ref{eq-nr1}) do not satisfy the estimate (7.4.5) in \cite{Han-Hong2006}.
As we have discussed,
the linearization of Monge-Amp\`{e}re
equations can be decomposed into two parts, one of which can be
used to form a linear equation whose solutions satisfy
the estimate (7.4.5) in \cite{Han-Hong2006} and another
which may be considered as quadratic error. Therefore the iteration process in the proof
of Theorem 7.4.1 can be performed.
We now outline the proof.

\begin{proof}[Proof of Theorem \ref{TheoremMain-Nonlinear}]
Now we can use iterations to solve $\mathcal F(\cdot,\varepsilon)=0$ for
$\varepsilon$ sufficiently small as in the proof of
Theorem 7.4.1 on page 130 \cite{Han-Hong2006}. The estimate
(\ref{eq-linear00}) plays the same role as (7.4.5) there.
We begin the iteration by setting $w_0=0$.  Then
$w_\ell$ is constructed by induction on $\ell$ as follows. Suppose $w_0,
w_1, \cdots , w_\ell$ have been chosen.
Let $\rho_\ell$ be a
solution of
\begin{equation}\label{B3.7}
{\mathcal{L}}(w_{\ell}) \rho_{\ell}=-\mathcal{F}(w_\ell).
\end{equation}
Here, $\rho_\ell$ is chosen to satisfy
\begin{equation*}\label{B3.9}
\|\rho_\ell\|_{H^s(B_1)}
\le c_s\left(||\mathcal{F}(w_\ell)||_{H^{s+d+3}(B_1)} +(\|w_\ell\|_{H^{s+d+8}(B_1)}+1)||\mathcal{F}(w_\ell)||_{H^{d+3}(B_1)}\right),
\end{equation*}
for any $s\ge 0$.
Now we define
\begin{equation}\label{B3.10}
w_{\ell+1}=w_{\ell}+ S_\ell\rho_{\ell}, \end{equation}
where $\{S_\ell\}$ is an appropriately chosen family of smoothing operators.
We point out that (\ref{B3.7}) replaces (7.4.8) in \cite{Han-Hong2006}.
We may proceed as in the proof of Theorem 7.4.1 \cite{Han-Hong2006}
with minor modifications. By Taylor expansion and
(\ref{B3.10}), we have
\begin{align*}
&\ \mathcal{F}(w_{\ell+1})-\mathcal{F}(w_\ell)\\
=&\
\mathcal{F}'(w_\ell)(w_{\ell+1}-w_{\ell})
+Q(w_\ell; w_{\ell+1}-w_{\ell})\\
=&\ \mathcal{F}'(w_\ell)(S_\ell \rho_{\ell}) +Q(w_\ell;
S_\ell\rho_{\ell})\\
=&\ \mathcal{L}(w_\ell)(S_\ell\rho_{\ell})+\big(\mathcal{F}'(w_\ell)
-\mathcal L(w_\ell)\big)
(S_\ell\rho_{\ell})+Q(w_\ell;
S_\ell\rho_{\ell})\\
=&\ \mathcal L(w_\ell)\rho_\ell+
\mathcal{L}(w_\ell)(S_\ell-1)\rho_\ell
+\big(\mathcal{F}'(w_\ell)
-\mathcal L(w_\ell)\big)
(S_\ell\rho_{\ell})+Q(w_\ell;
S_\ell\rho_{\ell}),\end{align*} where $Q(w_\ell; S_\ell\rho_{\ell})$
is the quadratic error. Then (7.4.23) on page 133 of \cite{Han-Hong2006}
may be modified accordingly.
We point out that, by (\ref{i2.19}) and (\ref{i2.20}),
the difference of ${\mathcal
F}'(w_\ell)
(S_\ell\rho_{\ell})$ and $\mathcal L(w_\ell)
(S_\ell\rho_{\ell})$
consists of {\it quadratic} expressions
in $\mathcal{F}(w_\ell)$ and $S_\ell\rho_{\ell}$,
and their derivatives, which may be estimated in a way similar to
$Q(w_\ell; S_\ell\rho_{\ell})$. The rest of the proof is the same
as that of Theorem 7.4.1 \cite{Han-Hong2006}, and is therefore not included here.
\end{proof}

\end{document}